\documentclass{amsart}
\usepackage{amsmath,amsfonts,amsthm,amssymb}
\usepackage{soul}
\usepackage{tikz}
\usetikzlibrary{arrows,decorations.pathmorphing,backgrounds,positioning,fit}
\usepackage{setspace}
\newtheorem{thm}{Theorem}[subsection]
\newtheorem*{thm*}{Theorem}

\newtheorem{lemma}[thm]{Lemma}

\newtheorem{coro}[thm]{Corollary}
\newtheorem{prop}[thm]{Proposition}
\theoremstyle{definition}
\newtheorem{defn}[thm]{Definition}
\theoremstyle{remark}
\newtheorem{rem}[thm]{Remark}
\newtheorem{ex}[thm]{Example}
\numberwithin{equation}{section}
\def\M{\mathbb{M}}
\def\tM{\widetilde{\mathbb{M}}}
\newcommand{\gmat}[2][ccccccccccccccccccccccccccccccccc]{\left[\begin{array}{#1} #2\\ \end{array}\right]}
\def\k{\mathbb{C}}
\def\S{\mathcal{S}}
\def\L{\mathcal{L}}
\def\V{\mathbb{V}}

\def\X{\mathbf{X}}
\newcommand{\oute}[2]{\mathbf{\Theta}_{#1,#2}}
\newcommand{\inv}[2]{\mathcal{O}_{#1,#2}}
\newcommand{\con}[2]{\mathcal{E}_{#1,#2}}
\def\Rep{\mathcal{O}Rep}
\def\Char{\mathcal{O}Char}

\title{Character algebras of decorated $SL_2(\k)$-local systems.}

\author{Greg Muller}

\author{Peter Samuelson}

\thanks{The first author was supported by the VIGRE program at LSU, National Science Foundation grant DMS-0739382.}

\begin{document}

\begin{abstract}
Let $\S$ be a path-connected, locally-compact CW-complex, and let $\M\subset \S$ be a 
subcomplex 
with finitely-many components.  A `decorated $SL_2(\k)$-local system' is an $SL_2(\k)$-local system on $\S$, together with a choice of `decoration' at each component of $\M$ (a section of the stalk of an associated vector bundle).  

We study the (decorated $SL_2(\k)$-)\emph{character algebra} of $(\S,\M)$, those functions on the space of decorated $SL_2(\k)$-local systems  on $(\S,\M)$ which are regular 
 with respect to 
the monodromy.
The character algebra 
is presented explicitly.  The character algebra is then shown to correspond to the algebra spanned by collections of oriented curves in $\S$ modulo simple graphical rules.

As an intermediate step we
obtain an invariant-theory result of independent interest: a presentation of the algebra of $SL_2(\k)$-invariant functions on $End(\V)^m\oplus \V^n$, where $\V$ is the tautological representation of $SL_2(\k)$. 
\end{abstract}

\maketitle


\section{Introduction}This paper seeks to generalize the theory of character algebras of $SL_2(\k)$-local systems by adding extra data (decorations) at a set $\M\subseteq \S$.

\subsection{$SL_2(\k)$-local systems}  The study of $SL_2(\k)$-local systems on a connected manifold $\S$ (or equivalently, maps $\pi_1(\S)\rightarrow SL_2(\k)$) is a well-known tool for studying the topology of $\S$.  	
This has its origins
in the study of hyperbolic 3-manifolds, where the hyperbolic metric can be used to construct an $SL_2(\mathbb{C})$-local system on $\S$.\footnote{The analogous construction for hyperbolic 2-manifolds requires a little bit more work.  There is a natural $PSL_2(\mathbb{R})$-local system, but it cannot be lifted to an $SL_2(\mathbb{C})$-local system without introducing `twisted local systems'; see \cite[Section 2.4]{FG06}.}

An $SL_2(\k)$-\emph{character function} of $\S$ (or just a `character function') is a $\k$-valued function on the set of equivalence classes of $SL_2(\k)$-local systems which is regular with respect to the monodromy.\footnote{Technically, this disregards the possibility of nilpotent character functions.  More correctly, a character function is a globally regular function on the moduli stack of $SL_2(\k)$-local systems.} For example, associated to any loop $l\in \pi_1(\S)$ there is a character $\chi_l$ which takes the trace of the monodromy of a local system around $l$.  The character functions collectively form the \emph{character algebra} $\Char(S)$.  The $\k$-points of the corresponding affine scheme $Char(S)=Spec(\Char(S))$ parametrize characters of $SL_2(\k)$-representations of $\pi_1(\S)$, and so a character function is a function on the set of characters.\footnote{Again, ignoring the issue of nilpotents.}

The character algebra was explored by Culler and Shalen \cite{CS83} to study incompressible surfaces in 3-manifolds.  Among other results, they proved that the character algebra is generated by characters of the form $\chi_l$, as $l$ runs over $\pi_1(\S)$.

In the 3-dimensional case, Bullock \cite{Bul97} (and later \cite{PS00}) showed that the relations amongst the characters $\{\chi_l\}$ can be encoded graphically, as the \emph{Kauffman skein relations} at $q=-1$.  These are relations coming from knot theory which are local, in the sense that they only manipulate a small neighborhood in $\S$.

A formal presentation of the character algebra was produced by Brumfiel and Hilden \cite{BH95} (building on the work of Procesi \cite{Pro76},\cite{Pro84}).  In fact, they present the character algebra of maps $G\rightarrow SL_2(\k)$ for an arbitrary finitely-generated $G$, generalizing the case when $G=\pi_1(\S)$.
\begin{thm}\cite[Proposition 9.1.i]{BH95}\label{thm: invBH}
For a finitely-generated group $G$, the $SL_2(\k)$-character algebra is the commutative algebra generated by $\chi_g$, for $g\in G$, together with relations
\begin{itemize}
\item $\chi_e=2$, and
\item $\chi_g\chi_h=\chi_{gh}+\chi_{g^{-1}h}$, for all $g,h\in G$.
\end{itemize}
\end{thm}
\subsection{Decorations} The theme of this paper is the addition of `decorations' to the preceding theory.  Fix a subset $\M\subseteq \S$, and let $\V$ be the tautological 2-dimensional representation of $SL_2(\k)$.  A \emph{decorated $SL_2(\k)$-local system} on $\S$ is...
\begin{itemize}
\item an $SL_2(\k)$-local system $\mathcal{L}$ on $\S$,
\item for each connected component $m\subset \M$, a choice of section $d_m$ over $m$ of the $\V$-vector bundle associated to $\mathcal{L}$.
\end{itemize}
Aside from being a very natural generalization of $SL_2(\k)$-local systems, there are several justifications for their study:
\begin{itemize}
\item Their character algebra can be explicitly presented, in a direct generalization of Theorem \ref{thm: invBH}.
\item Their character functions have a graphical calculus (Section 4) which generalizes the undecorated case.
\item They can be used to define `twisted decorated $SL_2(\k)$-local systems' on a surface, which have many excellent properties.  When $\S$ is hyperbolic, points in Penner's decorated Teichm\"uller space of $\S$ correspond to twisted decorated $SL_2(\k)$-local systems satisfying a positivity condition, see \cite{FG06}.
The character algebra of the twisted $SL_2(\k)$-local systems on $\S$ is naturally isomorphic to (the decorated version of) the `Kauffman skein algebra' at $q=1$.  Through this connection to skein algebras, the character algebra can be related to the cluster algebra of the surface (as defined in \cite{GSV05},\cite{FST08}).  Twisted decorated local systems will be explored in an upcoming paper \cite{MSSkein}.
\item In a physical theory where the role of wavefunction is played by $SL_2(\k)$-local systems (such as certain TQFTS, such as $\mathfrak{sl}_2(\k)$-Chern-Simons theory, see e.g. \cite{Ati90}), the decorations play the role of boundary values.  Here, the character functions are closely related to the observables; compare Section \ref{section: graphical} to the theory of `Wilson lines'.
\end{itemize}

\subsection{Character functions}
Like the undecorated case, decorated $SL_2(\k)$-local systems can be reduced to some algebraic \emph{monodromy data}.  However, instead a map of groups $\pi_1(\S)\rightarrow SL_2(\k)$, the monodromy data is a map of \emph{group actions} (see Section \ref{section: groupaction})
\[(\pi_1(\S),\pi_0(\tM))\rightarrow (SL_2(\k),\V)\]

A \emph{character function} has the analogous definition to the undecorated case; it is a function on isomorphism classes of decorated $SL_2(\k)$-local systems which is regular in the monodromy data.\footnote{As before, this disregards the possibility of nilpotent character functions.  More correctly, a character function is a globally regular function on the moduli stack of decorated $SL_2(\k)$-local systems.}   The character functions collectively form the \emph{character algebra} $\Char(\S,\M)$, our central object of study.

There are two easy examples of character functions on $(\S,\M)$, which will be informally referred to as the \emph{elementary character functions}.  
\begin{itemize}
\item First, for $l\in \pi_1(\S)$, the function $\chi_l$ which takes the trace of the monodromy of $\mathcal{L}$ around $l$ is a character function; note that this ignores the decorations.  
\item Second, if $a$ is a smooth, oriented arc connecting points $p,q\in \M$, then $\L$ and $a$ define a parallel transport map $\Gamma_a:\V_p\rightarrow \V_q$ between the stalks of the associated $\V$-vector bundle.  For a fixed bilinear skew form $\omega$ on $\V$, define
\[ \chi_a(\mathcal{L}^d):=\omega(\Gamma_a(d_p),d_q)\]
This character function only depends on $a$ up to homotopy (where the endpoints must stay in $\M$).
%
\end{itemize}
Let $\widetilde{\S}$ be the universal cover of $\S$, and let $\tM$ be the preimage of $\M$ in $\widetilde{\S}$.  Then a pair of components $p,q\in \pi_0(\tM)$ defines a homotopy class of oriented arcs $(p,q)$ in $(\S,\M)$, by considering the image in $\S$ of the unique class of oriented arcs connecting $p$ to $q$ in $\widetilde{\S}$.  In this way, a pair of components $p,q\in \pi_0(\tM)$ determines a character function $\chi_{(p,q)}$.

Then the main result of the paper is the following (Theorem \ref{thm: main}).
\begin{thm}
The character algebra $\Char(\S,\M)$ is the commutative $\k$-algebra generated by the elementary character functions $\chi_l$, for $l\in \pi_1(\S)$, and $\chi_{(p,q)}$, for $p,q\in \pi_0(\tM)$, with relations generated by
\begin{enumerate}
\item $\chi_e=2$, for $e$ the identity in $G$.
\item $\chi_{(p,q)}=-\chi_{(q,p)}$.
\item $\chi_{(gp,gq)}=\chi_{(p,q)}$.
\item $\chi_g\chi_h=\chi_{gh}+\chi_{g^{-1}h}$.
\item $\chi_g\chi_{(p,q)} = \chi_{(gp,q)}+\chi_{(g^{-1}p,q)}$.
\item $\chi_{(p,q)}\chi_{(p',q')}=\chi_{(p,q')}\chi_{(p',q)}+\chi_{(p,p')}\chi_{(q,q')}$.
\end{enumerate}
\end{thm}

Note that the generators all correspond to classes of oriented curves in $(\S,\M)$.  The relations in the theorem can then be translated into simple manipulations of these curves, and so computations in the character algebra may be performed graphically.

\subsection{The algebra of mixed $SL_2(\k)$-invariants} 
In the course of proving the main theorem, an invariant theory result of independent interest will be proven.  Recall $\V$ is the tautological $2$-dimensional representation of $SL_2(\k)$.  Let $End(\V)$ denote the algebra of linear endomorphisms of $\V$; it may be identified with the set of $2\times 2$ matrices.  The group $SL_2(\k)$ acts on $End(\V)$ by conjugation.

For $m,n\in \mathbb{N}$, we consider the vector space
$ End(\V)^m\oplus \V^n$
as an algebraic variety and denote its ring of functions by $\inv{m}{n}$.  
The invariant subalgebra $\inv{m}{n}^{SL_2}$ is called the algebra of \emph{mixed invariants}.\footnote{It interpolates between the algebra of invariant functions on $\V^n$ (a subject of classical invariant theory) and the algebra of invariant functions on $End(\V)$ (a subject of modern interest).}  This algebra appears to have first been considered by Procesi in \cite[Sec.12]{Pro76}.

Let $(\mathsf{A}_1,...,\mathsf{A}_m,v_1,...,v_n)$ denote an arbitrary element of $End(\V)^m\oplus \V^n$.  For $1\leq i\leq m$ and $1\leq j,k\leq n$, define
\[ \X_i(\mathsf{A}_1,...,\mathsf{A}_m,v_1,...,v_n):=\mathsf{A}_i\in End(\V)\]
\[ \oute{j}{k}(\mathsf{A}_1,...,\mathsf{A}_m,v_1,...,v_n):=v_jv_k^\perp\in End(\V)\]
For any word $\mathbf{A}$ in the $\{\X_i,\oute{j}{k}\}$, the function 
\[tr(\mathbf{A}):End(\V)^m\oplus \V^n\rightarrow \k\]
is $SL_2(\k)$-invariant.  Then the following is an explicit presentation of the algebra $\inv{m}{n}^{SL_2}$ of mixed invariants (Theorem \ref{thm: invpres}).
\begin{thm}
The algebra $\inv{m}{n}^{SL_2}$ of mixed invariants is the commutative algebra generated by $tr(\mathbf{A})$, as $\mathbf{A}$ runs over all words in the functions $\{\X_i,\oute{i}{j}\}$.
The relations are generated by
\begin{itemize}
\item (Procesi's F-relation) For $\mathbf{A},\mathbf{B},\mathbf{C}$ words in $\{\X_i,\oute{i}{j}\}$,
\begin{align*}
 tr(\mathbf{A}\mathbf{B}\mathbf{C})+tr(\mathbf{CBA})+tr(\mathbf{A})tr(\mathbf{B})tr(\mathbf{C}) &\\ =tr(\mathbf{B})tr(\mathbf{AC})&+tr(\mathbf{AB})tr(\mathbf{C})+ tr(\mathbf{A})tr(\mathbf{BC})
\end{align*}
\item
$tr(\mathbf{A}\oute{i}{j})=tr(\mathbf{A}\oute{j}{i})-tr(\mathbf{A})tr(\oute{j}{i})$, for $\mathbf{A}$ a word in the $\{\X_i,\oute{i}{j}\}$ and $1\leq i,j\leq n$,
\item $tr(\mathbf{A}\oute{i}{j}\mathbf{B}\oute{i'}{j'})=tr(\mathbf{A}\oute{i}{j'})tr(\mathbf{B}\oute{i'}{j})$, for $\mathbf{A},\mathbf{B}$ words in $\{\X_i,\oute{i}{j}\}$ and $1\leq i,j,i',j'\leq n$.
\end{itemize}
\end{thm}

\subsection{Structure of paper} 

The first half of the paper concerns definitions and the presentation of results.
\begin{itemize}
\item Section 2 defines local systems, decorated local systems, and their monodromy data.  We define the category of group actions, which is the proper context for the monodromy data.
\item Section 3 replaces the monodromy group action $(\pi_1(\S),\pi_0(\tM))$ with an arbitrary finitely-generated group action $(G,M)$, and defines the corresponding character algebra $\Char(G,M)$.  Basic results on elementary character functions are given, as well as a statement of the main theorem (Theorem \ref{thm: main}).
\item Section 4 explores the graphical nature of the character algebra.  A complete list of the rules for manipulating curves is presented, which is the graphical analog of Theorem \ref{thm: main}.
\end{itemize}
The second half of the paper proves the main result.  It is fairly computational and has a strong invariant-theory flavor. 
\begin{itemize}
\item Section 5 introduces the prerequisites from invariant theory, and defines the algebras of mixed invariants $\inv{m}{n}^{SL_2}$ and mixed concomitants $\con{m}{n}^{SL_2}$, which will be important intermediary objects.  Several previously-known results are presented as special cases of Theorem \ref{thm: invpres}. 
\item Section 6 produces presentations of the algebras of mixed invariants and mixed concomitants (Theorems \ref{thm: invpres} and \ref{thm: conpres}).
\item Section 7 applies the presentations from the previous section to produce a presentation of the character algebra $\Char(\S,\M)$.
\end{itemize}
The paper concluces with an appendix on twisted character algebras, a variant useful for applications.

\subsection{Acknowledgements} This paper owes its existence to many fellow mathematicians for encouraging discussions and directing the authors to the right track; in particular, the authors' PhD advisor Yuri Berest has given many patient, insightful explanations. Others deserving thanks include Adam Sikora, Dylan Thurston, and Milen Yakimov.  The authors are also grateful to the organizers of the conferences `Quantization of Singular Spaces' at Aarhus University (December 2010) and `Mathematical Aspects of Quantization' at Notre Dame (June 2011), during which much of this paper was proven and written.


\section{Decorated $SL_2(\k)$-local systems}

\subsection{Spaces with marked regions}  We begin by considering pairs $(\S,\M)$, where $\S$ is a path-connected, locally-compact CW-complex, and $\M\subset \S$ is an embedded CW-complex with finitely many components, each of which is path-connected ($\M$ may be empty).

The marked regions $\M$ will only matter up to homotopy equivalence in $\S$.  In particular, if each component of $\M$ is contractible in $\S$, then each component of $\M$ can be replaced by a point.  Regardless, $\pi_0(\M)$ will denote the set of components of $\M$.

Choose a basepoint $p\in \S$, and let $\pi_1(\S):=\pi_1(\S,p)$ be the fundamental group at $p$.   Let $\widetilde{\S}$ denote the universal cover of $\S$, 
and $\tM$ the preimage of $\M$.  The fundamental group $\pi_1(\S)$ then acts on $\widetilde{\S}$, $\tM$ and $\pi_0(\tM)$.

\subsection{$G$-local systems}

Recall that a \emph{$G$-local system} $\L$ on $\S$ is a locally-constant sheaf 
of sets
on $\S$ with a $G$-action, so that 
$\L$ is locally
a free and transitive $G$-set (a $G$-torsor).  Two $G$-local systems are equivalent if there is a $G$-equivariant isomorphism of sheaves between them.

The choice of a basepoint $p$ associates to every $G$-local system $\L$ its \emph{monodromy representation} $\rho_\L:\pi_1(\S)\rightarrow G$.  Two such representations are equivalent if they are $G$-conjugate.  Then the following is standard (see, e.g. \cite{Sza09}).
\begin{thm}\label{thm: local=mono}
The assignment $\L\rightarrow \rho_\L$ induces a bijection
\[ \left\{ \begin{array}{c}\text{$G$-local systems on $\S$}\\\text{up to equivalence}\end{array}\right\}
\stackrel{\sim}{\longrightarrow}
\left\{ \begin{array}{c}\text{homomorphisms $\pi_1(\S)\rightarrow G$}\\\text{up to equivalence}\end{array}\right\}\]
\end{thm}
In particular, if $\S$ is simply-connected, then every $G$-local system is trivial.

Given a $G$-local system $\L$ and a $G$-representation $V$, there is an \emph{induced vector bundle} $\L\times_G V$.  This is the vector bundle locally modeled on $V$ with transition maps determined by $\L$.

\subsection{Decorated $SL_2(\k)$-local systems}

Let $\V$ denote the standard $SL_2(\k)$- representation; it is 2-dimensional over $\k$.  Fix an $SL_2(\k)$-invariant, skew-symmetric bilinear form $\omega$ on $\V$, though this will only be used later.

\begin{defn}
A \textbf{decorated $SL_2(\k)$-local system} $\L^d$ on $(\S,\M)$ is...
\begin{itemize}
\item a $SL_2(\k)$-local system $\L$, and
\item for each connected component $m\subset \M$, a choice of local section $d_m\in (\L\times_{SL_2(\k)} \V)_m$.\footnote{Here, $(\L\times_{SL_2(\k)} \V)_m$ denotes the stalk of $\L\times_{SL_2(\k)} \V$ at $m$.}
\end{itemize}
\end{defn}

%


Note that (undecorated) $SL_2(\k)$-local systems are the special case when $\M=\emptyset$.  


Given a decorated $SL_2(\k)$-local system $\L^d$ on $(\S,\M)$, the universal covering map $\gamma:\widetilde{\S}\rightarrow \S$ induces a decorated $SL_2(\k)$-local system $\gamma^*\L^d$ on $(\widetilde{\S},\tM)$ (where $\tM=\gamma^{-1}(\M)$).  Since $\widetilde{\S}$ is simply-connected, $\gamma^*\L^d$ can be globally trivialized so that
\[\Gamma(\widetilde{\S},\gamma^*\L^d)\simeq SL_2(\k),\;\;\; \Gamma(\widetilde{\S},(\gamma^*\L^d\times_{SL_2(\k)}\V))\simeq \V\]
The action of $\pi_1(\S)$ on $\widetilde{\S}$ gives an action on local sections which recovers the monodromy map $\rho:\pi_1(\S)\rightarrow SL_2(\k)$.  Since all the stalks in $\gamma^*\L^d$ may be identified with the global sections, the decorations on $\tM$ amount to a map $\rho_d:\pi_0(\tM)\rightarrow \V$.  If two marked components $m,m'\in \M$ are related by the action of $l\in \pi_1(\S)$, then their decorations must be related by the corresponding monodromy; that is, $\forall m\in \M$ and $l\in \pi_1(\S)$,
\[ \rho_d(lm) =\rho(l)\rho_d(m)\]
Changing the global trivialization of $\gamma^*\L^d$ by $g\in SL_2(\k)$ will send $\rho$ to $g\rho g^{-1}$, and $\rho_d$ to $g\rho_d$.

\subsection{Group actions}\label{section: groupaction}

This monodromy information is a specific example of a more general algebraic structure.  Define a \textbf{group action} to be a pair $(G,M)$, where $G$ is a group and $M$ is a $G$-set.  A morphism of group actions is a pair of maps
\[(f_G,f_M):(G,M)\rightarrow (G',M')\]
such that $f_G:G\rightarrow G'$ is a group homomorphism and $f_M:M\rightarrow M'$ is a set map such that, $\forall g\in G$ and $m\in M$
\[ f_M(gm)=f_G(g)f_M(m)\]
A group action $(G,M)$ is \textbf{finitely-generated} if $G$ is finitely-generated as a group and $M$ has finitely many $G$-orbits.  

Any element $g\in G$ determines an \emph{inner} isomorphism of group actions $(ad_g,g\cdot):(G,M)\rightarrow (G,M)$.  Then $G$ acts on the set
\[ Hom( (G',M'),(G,M))\]
by 
postcomposition with the
corresponding inner automorphism.

Applying this to the previous setting, choosing a global trivialization of $\gamma^*\L^d$ gives a morphism of group actions
\begin{equation}
(\rho,\rho_d):(\pi_1(\S),\pi_0(\tM))\rightarrow (SL_2(\k),\V)
\end{equation}
called the (decorated) \textbf{monodromy representation}.  Two different global trivializations give monodromy representations related by the action of $SL_2(\k)$.
\begin{prop}\label{prop: mono}
Monodromy representations determine a bijecton
\[ \left\{ \begin{array}{c}\text{decorated $SL_2(\k)$-local systems}\\\text{ on $(\S,\M)$ up to equivalence}\end{array}\right\}
\stackrel{\sim}{\longrightarrow}
\left\{ \begin{array}{c}\text{group action maps}\\ (\pi_1(\S),\pi_0(\tM))\rightarrow (SL_2(\k),\V)\\ \text{up to the action of $SL_2(\k)$}\end{array}\right\}\]
\end{prop}
\begin{proof}
The map from decorated local systems to group action maps has already been described; we now construct its inverse.  Given a representation \[f:(\pi_1(\S),\pi_0(\tM))\rightarrow (SL_2(\k),\V),\] the map $f_{\pi_1(\S)}:\pi_1(\S)\rightarrow SL_2(\k)$ determines an $SL_2(\k)$-local system $\mathcal{L}$, by Theorem \ref{thm: local=mono}.  

The map $f_{\pi_0(\tM)}:\pi_0(\tM)\rightarrow \V$ determines a section of the stalk $\widetilde{d}\in(\mathcal{L}\times_{SL_2(\k)}\V)_{\tM}$.  For $U\subset \S$, sections of $\mathcal{L}\times_{SL_2(\k)}\V$ over $U$ can be identified with $\pi_1(\S)$-equivariant sections of $\gamma^{-1}U\subset \widetilde{\S}$.   The group action condition implies that $\widetilde{d}$ is $\pi_1(\S)$-equivariant, so it determines an element $d\in (\mathcal{L}\times_{SL_2(k)}\V)_\M$.  This defines a decoration of $\mathcal{L}$.
\end{proof}

\begin{rem}\emph{Naturality.}
This bijection is natural, in that it can be made into an equivalence of the appropriate categories of families.  As a consequence, this induces an isomorphism between the corresponding moduli stacks.
\end{rem}

\begin{rem}\emph{General decorated local systems.}
The definition of decorated local systems works the same if $SL_2(\k)$ is replaced by any group $G$, and $\V$ is replaced by any $G$-set $M$.
\end{rem}

\section{Character algebras}

The decorated $SL_2(\k)$-local systems on $(\S,\M)$ only depend on the finitely-generated group action $(\pi_1(\S), \pi_0(\tM))$.	  This section will pass to the larger generality of an arbitrary finitely-generated group action $(G,M)$, with $(\pi_1(\S),\pi_0(\tM))$ demoted to the role of `motivating example'.

%

\subsection{Algebraic structure}

The first step is to endow the set of group action maps\footnote{These will often be informally referred to as \emph{representations} of $(G,M)$.} $(G,M)\rightarrow (SL_2(\k),\V)$ with the structure of an affine scheme.
\begin{defn}
The \textbf{representation algebra} of $(G,M)$ into $(SL_2(\k),\V)$ is the $\k$-algebra $\Rep(G,M)$ such that, for any commutative $\k$-algebra $A$, there is a  bijection (functorial in $A$)
\[ \left\{ \begin{array}{c}\text{group action maps}\\ (G,M)\rightarrow (SL_2(A),A\otimes \V)\end{array}\right\}
\stackrel{\sim}{\longrightarrow}
\left\{ \begin{array}{c}\text{$\k$-algebra maps}\\ \Rep(G,M)\rightarrow A\end{array}\right\}\]
The affine scheme $Rep(G,M):=Spec(\Rep(G,M))$ is called the \textbf{representation scheme}.
\end{defn}
In particular, the $\k$-points of $Rep(G,M)$ are in bijection with the set of group action maps $(G,M)\rightarrow (SL_2(\k),\V)$. The proof of the existence and uniqueness of $\Rep(G,M)$ will be deferred to Section \ref{section: repalg} (specifically, Proposition \ref{prop: unirep}).
%
%

The group $SL_2(\k)$ acts naturally on $\Rep(G,M)$, and the set of $SL_2(\k)$-equivalence classes of maps $(G,M)\rightarrow (SL_2(\k),\V)$ is the set of $SL_2(\k)$-orbits in $Rep(G,M)$.

\subsection{The character algebra}

\begin{defn}
Define the \textbf{character algebra} of $(G,M)$ into $(SL_2(\k),\V)$ by
\[\Char(G,M):=\Rep(G,M)^{SL_2(\k)}\subset \Rep(G,M),\] the $SL_2(\k)$-invariant subalgebra of $\Rep(G,M)$.  The \textbf{character scheme} is the affine scheme $Char(G,M):=Spec(\Char(G,M))$; it is the algebro-geometric quotient of $Rep(G,M)$ by $SL_2(\k)$.
\end{defn}
A \textbf{character function} is an element of $\Char(G,M)$.  Intuitively, it is an $SL_2(\k)$-invariant function on the set of group action maps $(G,M)\rightarrow (SL_2(\k),\V)$ which can be expressed as a polynomial in coordinates.\footnote{However, this intuition ignores the possibility of nilpotent elements in $\Char(G,M)$.} That is, an element of $\chi\in \Char(\pi_1(\S),\pi_0(\tM))$ assigns to every decorated $SL_2(\k)$-local system $\L^d$ a number $\chi(\L^d)\in \k$, so that $\chi$ varies algebraically with respect to $\L^d$.

\begin{rem}\emph{(Decorated characters?)} When $\M\neq \emptyset$, the $\k$-points of $Char(G,M)$ do not correspond to characters of some group.  The best one can say is as follows.  The $\k$-points of $Char(\S,\M)$ correspond to equivalence classes of the weakest Zariski-closed equivalence relation on the set of maps $(G,M)\rightarrow (SL_2(\k),\V)$ 
such that conjugate pairs are contained in equivalence classes.
\end{rem}

\begin{rem}\emph{(The moduli stack)}
The space of $SL_2(\k)$-orbits in $Rep(G,M)$ is naturally an algebraic stack.
This stack usually won't be representable by a scheme.  However, it still has a universal `schemification' - a scheme through which all maps to schemes factor.  This schemification is the character scheme $Char(G,M)$.
\end{rem}

\subsection{Elementary character functions}

There are two immediate sources of character functions.  First, let $g\in G$.  For a map of group actions
\[ \rho:=(\rho_G,\rho_M):(G,M)\rightarrow (SL_2(\k),\V),\]
define
\begin{equation}
\chi_g(\rho):= tr(\rho_G(g))
\end{equation}
For $a\in SL_2(\k)$,
\[\chi_g(a\cdot \rho)=tr(a\rho(g)a^{-1})=tr(\rho(g))=\chi_g(\rho)\]
and so $\chi_g$ defines a character function for all $g\in G$.

Next, let $p,q\in M$.  Recall that $\omega$ is a fixed $SL_2(\k)$-invariant, skew-symmetric bilinear form on $\V$. For
$\rho$ an arbitrary representation, define
\begin{equation} \chi_{(p,q)}(\rho):= \omega(\rho_M(p),\rho_M(q))
\end{equation}
For $a\in SL_2(\k)$,
\[\chi_{(p,q)}(a\cdot \rho)=\omega(a\rho_M(p),a\rho_M(q))=\omega(\rho_M(p),\rho_M(q))=\chi_{(p,q)}(\rho)\]
and so $\chi_{(p,q)}$ defines a character function for all $g\in G$.

\subsection{Relations between elementary character functions.}

\begin{prop}\label{prop: relations}
The following relations hold between character functions of the form $\chi_g$ and $\chi_{(p,q)}$.
\begin{enumerate}
\item $\chi_e=2$, for $e$ the identity in $G$.
\item $\chi_{(p,q)}=-\chi_{(q,p)}$.
\item $\chi_{(gp,gq)}=\chi_{(p,q)}$.
\item $\chi_g\chi_h=\chi_{gh}+\chi_{g^{-1}h}$.
\item $\chi_g\chi_{(p,q)} = \chi_{(gp,q)}+\chi_{(g^{-1}p,q)}$.
\item $\chi_{(p,q)}\chi_{(p',q')}=\chi_{(p,q')}\chi_{(p',q)}+\chi_{(p,p')}\chi_{(q,q')}$.
\end{enumerate}
\end{prop}
\begin{proof}Relation $(1)$ is $tr(Id_{\V})=\dim(\V)=2$.  Relation $(2)$ is the anti-symmetry of $\omega$. Relation $(3)$ is the $SL_2(\k)$-invariance of $\omega$.  Relations $(4)$ and $(5)$ use the following lemma.
\begin{lemma}
Let $g\in SL_2(\k)$.  Then $g+g^{-1}=tr(g)\cdot Id_\V$.
\end{lemma}
\begin{proof}[Proof of lemma]
The matrix $g$ satisfies its own characteristic polynomial; that is,
\[ g^2-tr(g)\cdot g + Id_\V=0\]
Multiplying by $g^{-1}$ gives the desired identity.
\end{proof}
Then Relation $(4)$ is
\begin{equation}
tr(g)tr(h) = tr(tr(g)h) = tr(gh+g^{-1}h) = tr(gh)+tr(g^{-1}h)
\end{equation}
and Relation $(5)$ is
\begin{equation}
tr(g)\omega(p,q)=\omega(tr(g)p,q) = \omega(gp+g^{-1}p,q)=\omega(gp,q)+\omega(g^{-1}p,q)
\end{equation}
Relation (6) is equivalent to the `quadratic Pl\"ucker relation'.
To see this, choose $\omega$-canonical coordinates $x,y$ on $\V$; that is, 
\[\omega(p,q)=x_py_q-x_qy_p=\left|\begin{array}{cc} x_p & x_q \\ y_p & y_q\end{array}\right|,\forall p,q\in \V\]
Then Relation (6) becomes
\[
\left|\begin{array}{cc} x_{p} & x_{q} \\ y_{p} & y_{q}\end{array}\right|
\left|\begin{array}{cc} x_{p'} & x_{q'} \\ y_{p'} & y_{q'}\end{array}\right|=
\left|\begin{array}{cc} x_{p} & x_{q'} \\ y_{p} & y_{q'}\end{array}\right|
\left|\begin{array}{cc} x_{p'} & x_{q} \\ y_{p'} & y_{q}\end{array}\right|+
\left|\begin{array}{cc} x_{p} & x_{p'} \\ y_{p} & y_{p'}\end{array}\right|
\left|\begin{array}{cc} x_{q} & x_{q'} \\ y_{q} & y_{q'}\end{array}\right|
\]
which can be verified directly.
\end{proof}

These relations formally imply several other relations.
\begin{coro}\label{coro: relations}
The following relations also hold between character functions of the form $\chi_g$ and $\chi_{(p,q)}$.
\begin{enumerate}
\setcounter{enumi}{6}
\item $\chi_{p,p}=0$.
\item $\chi_{g^{-1}}=\chi_g$.
\item $\chi_{hgh^{-1}}=\chi_g$.
\item $\chi_{g^i}=2T_i(\chi_g/2)$, where $T_i$ is the \emph{$i$th Chebyshev polynomial}.\footnote{These are defined by $T_0(x)=1$, $T_1(x)=x$ and $T_{i+1}(x)=2xT_i(x)-T_{i-1}(x)$.}
\end{enumerate}
\end{coro}

\subsection{A presentation of $\Char(G,M)$}

The main result of this paper will be that the elementary character functions $\{\chi_g\}$ and $\{\chi_{(p,q)}\}$ generate $\Char(G,M)$, with all relations generated by those in Proposition \ref{prop: relations}.  We formalize this as follows.

\begin{defn}\label{defn: H^+(G,M)}
Let $H^+(G,M)$ denote\footnote{The notation $H^+(G,M)$ is chosen to match \cite{BH95} and \cite{PS00}.} the commutative $\k$-algebra generated by symbols
\begin{itemize}
\item $[g]$, for all $g\in G$, and
\item $[p,q]$, for all $p,q\in M$,
\end{itemize}
with relations generated by...
\begin{enumerate}
\item $[e]-2$.
\item $[p,q]+[q,p]$.
\item $[gp,gq]-[p,q]$.
\item $[g][h]-[gh]-[gh^{-1}]$.
\item $[g][p,q]-[gp,q]-[p,gq]$.
\item $[p,q][p',q']-[p,q'][p',q]-[p,p'][q,q']$.
\end{enumerate}
\end{defn}
By Proposition \ref{prop: relations}, there is an algebra map
\[ \chi:H^+(G,M)\rightarrow \Char(G,M)\]
such that $\chi([g])=\chi_g$ and $\chi([p,q])= \chi_{(p,q)}$.

\begin{thm}\label{thm: main}
The map $\chi:H^+(G,M)\rightarrow \Char(G,M)$ is an isomorphism.
\end{thm}
\begin{proof}[Proof Outline] The proof will be contained in Sections 5-7, but we outline the idea here.   If $G$ is a free group on $m$ generators and $M$ is a free $G$-set of $n$ orbits, then $Rep(G,M)=SL_2(\k)^m\oplus \V^n$.  This is an $SL_2$-fixed subvariety of the affine space $End(\V)^m\oplus \V^n$.  The invariant functions $\inv{m}{n}^{SL_2}$ on $End(\V)^m\oplus \V^n$ were studied in \cite[Section 12]{Pro76} under the name \emph{mixed invariants}.  A nice presentation for the mixed invariants is obtained (Theorem \ref{thm: invpres}), which can be used to present $\Char(G,M)$ when $G$ and $M$ are both free (see Remark \ref{rem: gettinginvariant}).

When $G$ and $M$ are not free, a choice of $m$ generators for $G$ and $n$ generators for $M$ realizes $Rep(G,M)$ as an $SL_2$-fixed subvariety of $End(\V)^m\oplus \V^n$, corresponding to the image of the generators in $SL_2(\k)$ and $\V$, respectively.  The problem that arises is that the natural relations defining $Rep(G,M)$ in $End(\V)^m\oplus \V^n$ are not $SL_2$-invariant, and so more relations might appear in the invariant subalgebra.

This is solved by considering the algebra $\con{m}{n}^{SL_2}$ of $SL_2$-equivariant maps from $End(\V)^m\oplus \V^n$ to $End(\V)$, the algebra of \emph{mixed concomitants}.  When extended to this non-commutative algebra, the defining equations of $Rep(G,M)$ become $SL_2$-invariant, and so it is possible to present $(End(\V)\otimes \Rep(G,M))^{SL_2}$ (Theorem \ref{thm: matcharpres}).  Then, by taking traces of these elements, we present $\Char(G,M)$.
\end{proof}

\begin{ex}When $M=\emptyset$, then Relations (2), (3), (5), and (6) are vacuous, and the resulting presentation coincides with Brumfiel and Hilden's presentation (Theorem \ref{thm: invBH}).
\end{ex}
\begin{ex}When $G=\{e\}$ and $M=\{1,...,n\}$, then Relations (3), (4) and (5) are redundant, and the character algebra is the homogeneous coordinate ring of the Grassmannian $Gr(2,n)$.  The Pl\"ucker coordinates $x_{ij}$ correspond to $\chi_{(i,j)}$.
\end{ex}
\begin{rem}
These last two examples can be interpreted as saying that general character algebras $\Char(G,M)$ interpolate between Brumfiel and Hilden's $SL_s(\k)$-representation algebras $H^+(G)$ and the homogeneous coordinate rings of Grassmannians.
\end{rem}
\begin{ex}
For $G=\mathbb{Z}$ and $M=\{a,b\}$ (two $G$-invariant elements), the algebra $\Char(G,M)$ is generated by $\chi_1$ and $\chi_{(a,b)}$, with the single relation $(\chi_1+2)\chi_{(a,b)}=0$.  The remaining elementary character functions can be expressed in terms of these two; for example, $\chi_{i}=2T_i(\chi_1/2)$ and $\chi_{(b,a)}=-\chi_{(a,b)}$.

The corresponding scheme $Char(G,M)$ is two affine lines crossing transversely.  The two irreducible components of this scheme correspond to maps $\rho:(G,M)\rightarrow(SL_2(\k),\V)$ such that either
\begin{itemize}
\item $\rho_G(1)=Id_\V$ (when $\chi_1=tr(\rho(1))=2$), or 
\item $\rho_M(a)$ and $\rho_M(b)$ are linearly-dependant (when $\chi_{(a,b)}=\omega(\rho(a),\rho(b))=0$).
\end{itemize}
Since $\rho_M(a)$ and $\rho_M(b)$ are invariant vectors of $\rho_G(1)$, it is clear one of these two conditions must be satisfied.
\end{ex}

\section{The graphical nature of the character algebra}\label{section: graphical}

In this section, we return to the topological setting of $(G,M)=(\pi_1(\S),\pi_0(\tM))$, and hence to character functions for decorated $SL_2(\k)$-local systems.  The elementary character functions naturally correspond to homotopy classes of oriented curves in $\S$, and the relations from the previous section can be graphically encoded.  

\subsection{Character functions of oriented curves}

A \emph{curve} in $(\S,\M)$ will be a closed immersion $c:\mathcal{C}\rightarrow \S$ of a connected, 1-dimensional manifold with boundary, such that any endpoints of $\mathcal{C}$ land in $\M$.  An \emph{orientation} of $c$ is an orientation of $\mathcal{C}$.  

There are two kinds of curves in $(\S,\M)$.
\begin{itemize}
\item \emph{Loops} $l$, where $\mathcal{C}$ is homeomorphic to the circle $S^1$,
\item \emph{Arcs} $a$, where $\mathcal{C}$ is homeomorphic to the interval $[0,1]$ and $c(0),c(1)\in \M$.
\end{itemize}

A character function $\chi_c\in \Char(\S,\M)$ can be associated to any oriented curve $c$, as follows.
\begin{itemize}
\item If $c$ is an oriented loop $l$, choose a basepoint $p\in l$.  Then $l$ determines an element $\pi_1(\S,p)$, and so it determines a character function $\chi_{l}$.  This function is well-defined, since $\chi_l=\chi_{glg^{-1}}$.  It is the trace of the monodromy around $l$.
\item If $c$ is an oriented arc $a$, choose a lift $\widetilde{a}$ of $a$ to the universal cover $\widetilde{\S}$.  Let $p,q\in \tM$ be the endpoints of $\widetilde{a}$; this determines a character function $\chi_a:=\chi_{p,q}$.  Note that a different lift of $a$ will have endpoints $gp$ and $gq$ (for some $g$), and so $\chi_a$ is well-defined.  This character function amounts to using $a$ to parallel transport the decoration from one endpoint to the other, and using $\omega$ to compare the decorations.
\end{itemize}
In both cases, the character function $\chi_c$ only depends on $c$ up to homotopy (where the endpoints are required to stay in $\M$).  

For a finite collection of oriented curves $\mathbf{c}=\{c_1,c_2,...,c_i\}$, the associated character function is the product of the corresponding elementary character functions.
\[ \chi_{\mathbf{c}} = \chi_{c_1}\chi_{c_2}...\chi_{c_i}\]

\subsection{Graphical relations} Let $Curve(\S,\M)$ denote the $\k$-algebra spanned by the set of finite collections of oriented curves in $(\S,\M)$ (up to homotopy).  Then 
\[\chi:Curve(\S,\M)\rightarrow \Char(\S,\M)\] defines a map of algebras.  Since every elementary character function can be written in the form $\chi_c$, Theorem \ref{thm: main} shows that $\Char(\S,\M)$ is a quotient of $Curve(\S,\M)$ with kernel determined by the relations in $\Char(\S,\M)$.  The generators of this kernel can be interpreted as graphical rules; this gives a complete set of rules for manipulating collections of curves without changing the corresponding element of the character algebra.


\begin{center} \underline{Graphical rules for characters of oriented curves}\end{center}
\begin{enumerate}
\item A contractible loop is equal to $2$ (Proposition \ref{prop: relations}, $(1)$). 
\[
\begin{tikzpicture}[inner sep=0.5mm,scale=1,auto]
	\draw[->] (0,0) arc (0:360:.2);
\end{tikzpicture}
=2\]

\item A contractible arc is  equal to $0$ (Corollary \ref{coro: relations}, $(7)$).
\[
\begin{tikzpicture}[inner sep=0.5mm,scale=1,auto]
	\node (p) at (0,0) {$\circ$};
	
	\draw[->] (0,0) to [out=-45,in=-90] (.6,0);
	\draw (.6,0) to [out=90, in=45] (0,0);
\end{tikzpicture}
=0\]

\item An oriented loop is equal to its orientation-reversal (Corollary \ref{coro: relations}, $(8)$).
\[
\begin{tikzpicture}[inner sep=0.5mm,scale=1,auto]
	\node (3) at (-.25,.35)  {};
	\node (4) at (-.25,-.35)  {};
	
	\draw[->] (4) arc (-45:0:.5);
	\draw (-.1,0) arc (0:45:.5);
	\draw[dashed] (3) arc (45:315:.5);
\end{tikzpicture}
=
\begin{tikzpicture}[inner sep=0.5mm,scale=1,auto]
	\node (3) at (-.25,.35)  {};
	\node (4) at (-.25,-.35)  {};
	
	\draw (4) arc (-45:0:.5);
	\draw[<-] (-.1,0) arc (0:45:.5);
	\draw[dashed] (3) arc (45:315:.5);
\end{tikzpicture}
\]

\item An oriented arc is negative its orientation-reversal (Proposition \ref{prop: relations}, $(2)$).
\[
\begin{tikzpicture}[inner sep=0.5mm,scale=1,auto]
	\node (3) at (-.25,.35)  {};
	\node (4) at (-.25,-.35)  {};
	\node (p) at (-.95,.35) {$\circ$};
	\node (q) at (-.95,-.35) {$\circ$};
	
	\draw[->] (4) arc (-45:0:.5);
	\draw (-.1,0) arc (0:45:.5);
	\draw[dashed] (3) arc (45:135:.5);
	\draw[dashed] (4) arc (-45:-135:.5);
\end{tikzpicture}
= - 
\begin{tikzpicture}[inner sep=0.5mm,scale=1,auto]
	\node (3) at (-.25,.35)  {};
	\node (4) at (-.25,-.35)  {};
	\node (p) at (-.95,.35) {$\circ$};
	\node (q) at (-.95,-.35) {$\circ$};
	
	\draw (4) arc (-45:0:.5);
	\draw[<-] (-.1,0) arc (0:45:.5);
	\draw[dashed] (3) arc (45:135:.5);
	\draw[dashed] (4) arc (-45:-135:.5);
\end{tikzpicture}
\]

\item There are relations for pairs of curves, by choosing nearby segments on each curve and summing over the two other ways to connect them.  

Two loops (Proposition \ref{prop: relations}, $(4)$).
\[
\begin{tikzpicture}[inner sep=0.5mm,scale=1,auto]
	\node (1) at (.25,.35)  {};
	\node (2) at (.25,-.35)  {};
	\node (3) at (-.25,.35)  {};
	\node (4) at (-.25,-.35)  {};
	
	\draw[->] (4) arc (-45:0:.5);
	\draw (-.1,0) arc (0:45:.5);
	\draw[dashed] (3) arc (45:315:.5);
	\draw[->] (2) arc (225:180:.5);
	\draw (.1,0) arc (180:135:.5);
	\draw[dashed] (1) arc (135:-135:.5);
\end{tikzpicture}
=
\begin{tikzpicture}[inner sep=0.5mm,scale=1,auto]
	\node (1) at (.25,.35)  {};
	\node (2) at (.25,-.35)  {};
	\node (3) at (-.25,.35)  {};
	\node (4) at (-.25,-.35)  {};
	
	\draw[->] (4) to [out=45,in=180] (0,-.15);
	\draw (3) to [out=-45,in=180] (0,.15);
	\draw[dashed] (3) arc (45:315:.5);
	\draw[->] (1) to [out=225,in=0] (0,.15);
	\draw (2) to [out=135,in=0] (0,-.15);
	\draw[dashed] (1) arc (135:-135:.5);
\end{tikzpicture}
+
\begin{tikzpicture}[inner sep=0.5mm,scale=1,auto]
	\node (1) at (.25,.35)  {};
	\node (2) at (.25,-.35)  {};
	\node (3) at (-.25,.35)  {};
	\node (4) at (-.25,-.35)  {};
	
	\draw[->] (4) to [out=45,in=225] (.13,.20);
	\draw (3) to [out=-45,in=135] (-.13,.20);
	\draw[dashed] (3) arc (45:315:.5);
	\draw (1) to [out=225,in=45] (.13,.20);
	\draw[->] (2) to [out=135,in=-45] (-.13,.20);
	\draw[dashed] (1) arc (135:-135:.5);
\end{tikzpicture}
\]

\item An arc and a loop (Proposition \ref{prop: relations}, $(5)$).
\[
\begin{tikzpicture}[inner sep=0.5mm,scale=1,auto]
	\node (1) at (.25,.35)  {};
	\node (2) at (.25,-.35)  {};
	\node (3) at (-.25,.35)  {};
	\node (4) at (-.25,-.35)  {};
	\node (p) at (-.95,.35) {$\circ$};
	\node (q) at (-.95,-.35) {$\circ$};
	
	\draw[->] (4) arc (-45:0:.5);
	\draw (-.1,0) arc (0:45:.5);
	\draw[dashed] (3) arc (45:135:.5);
	\draw[dashed] (4) arc (-45:-135:.5);
	\draw[->] (2) arc (225:180:.5);
	\draw (.1,0) arc (180:135:.5);
	\draw[dashed] (1) arc (135:-135:.5);
\end{tikzpicture}
=
\begin{tikzpicture}[inner sep=0.5mm,scale=1,auto]
	\node (1) at (.25,.35)  {};
	\node (2) at (.25,-.35)  {};
	\node (3) at (-.25,.35)  {};
	\node (4) at (-.25,-.35)  {};
	\node (p) at (-.95,.35) {$\circ$};
	\node (q) at (-.95,-.35) {$\circ$};
	
	\draw[->] (4) to [out=45,in=180] (0,-.15);
	\draw (3) to [out=-45,in=180] (0,.15);
	\draw[dashed] (3) arc (45:135:.5);
	\draw[dashed] (4) arc (-45:-135:.5);
	\draw[->] (1) to [out=225,in=0] (0,.15);
	\draw (2) to [out=135,in=0] (0,-.15);
	\draw[dashed] (1) arc (135:-135:.5);
\end{tikzpicture}
+
\begin{tikzpicture}[inner sep=0.5mm,scale=1,auto]
	\node (1) at (.25,.35)  {};
	\node (2) at (.25,-.35)  {};
	\node (3) at (-.25,.35)  {};
	\node (4) at (-.25,-.35)  {};
	\node (p) at (-.95,.35) {$\circ$};
	\node (q) at (-.95,-.35) {$\circ$};
	
	\draw[->] (4) to [out=45,in=225] (.13,.20);
	\draw (3) to [out=-45,in=135] (-.13,.20);
	\draw[dashed] (3) arc (45:135:.5);
	\draw[dashed] (4) arc (-45:-135:.5);
	\draw (1) to [out=225,in=45] (.13,.20);
	\draw[->] (2) to [out=135,in=-45] (-.13,.20);
	\draw[dashed] (1) arc (135:-135:.5);
\end{tikzpicture}
\]

\item Two arcs (Proposition \ref{prop: relations}, $(6)$).
\[
\begin{tikzpicture}[inner sep=0.5mm,scale=1,auto]
	\node (1) at (.25,.35)  {};
	\node (2) at (.25,-.35)  {};
	\node (3) at (-.25,.35)  {};
	\node (4) at (-.25,-.35)  {};
	\node (p) at (-.95,.35) {$\circ$};
	\node (q) at (-.95,-.35) {$\circ$};
	\node (p') at (.95,.35) {$\circ$};
	\node (q') at (.95,-.35) {$\circ$};
	
	\draw[->] (4) arc (-45:0:.5);
	\draw (-.1,0) arc (0:45:.5);
	\draw[dashed] (3) arc (45:135:.5);
	\draw[dashed] (4) arc (-45:-135:.5);
	\draw[->] (2) arc (225:180:.5);
	\draw (.1,0) arc (180:135:.5);
	\draw[dashed] (1) arc (135:45:.5);
	\draw[dashed] (2) arc (-135:-45:.5);
\end{tikzpicture}
=
\begin{tikzpicture}[inner sep=0.5mm,scale=1,auto]
	\node (1) at (.25,.35)  {};
	\node (2) at (.25,-.35)  {};
	\node (3) at (-.25,.35)  {};
	\node (4) at (-.25,-.35)  {};
	\node (p) at (-.95,.35) {$\circ$};
	\node (q) at (-.95,-.35) {$\circ$};
	\node (p') at (.95,.35) {$\circ$};
	\node (q') at (.95,-.35) {$\circ$};
	
	\draw[->] (4) to [out=45,in=180] (0,-.15);
	\draw (3) to [out=-45,in=180] (0,.15);
	\draw[dashed] (3) arc (45:135:.5);
	\draw[dashed] (4) arc (-45:-135:.5);
	\draw[->] (1) to [out=225,in=0] (0,.15);
	\draw (2) to [out=135,in=0] (0,-.15);
	\draw[dashed] (1) arc (135:45:.5);
	\draw[dashed] (2) arc (-135:-45:.5);
\end{tikzpicture}
+
\begin{tikzpicture}[inner sep=0.5mm,scale=1,auto]
	\node (1) at (.25,.35)  {};
	\node (2) at (.25,-.35)  {};
	\node (3) at (-.25,.35)  {};
	\node (4) at (-.25,-.35)  {};
	\node (p) at (-.95,.35) {$\circ$};
	\node (q) at (-.95,-.35) {$\circ$};
	\node (p') at (.95,.35) {$\circ$};
	\node (q') at (.95,-.35) {$\circ$};
	
	\draw[->] (4) to [out=45,in=225] (.13,.20);
	\draw (3) to [out=-45,in=135] (-.13,.20);
	\draw[dashed] (3) arc (45:135:.5);
	\draw[dashed] (4) arc (-45:-135:.5);
	\draw (1) to [out=225,in=45] (.13,.20);
	\draw[->] (2) to [out=135,in=-45] (-.13,.20);
	\draw[dashed] (1) arc (135:45:.5);
	\draw[dashed] (2) arc (-135:-45:.5);
\end{tikzpicture}
\]

\end{enumerate}

\begin{rem}
Several remarks are due.  
\begin{enumerate}
\item The relation $\chi_{gp,gq}=\chi_{p,q}$ appears to be missing; it is implicit in the definition of the character function of a curve. 
\item The second and third rule are formal consequences of the other rules.  
\item While the last three rules appear to be virtually identical, there is an important distinction.  The orientation of arcs in the last rule is arbitrary-seeming but necessary, whereas the orientation of loops in the fifth rule is truly arbitrary.
\item Despite how the last three rules have been drawn, it is possible that the curves on the left hand side intersect along the chosen segments. 
\end{enumerate}
\end{rem}

\subsection{Weaknesses of this approach}

While it might be tempting to regard this as the `correct way' to graphically visualize the character algebra, there are two shortcomings of this approach.
\begin{itemize}
\item \emph{Orientation-dependance:} The function $\chi_c$ depends on the orientation of a curve $c$, but only up to a sign.
\item \emph{Non-local relations:} The relations will not be local in $\S$.  For each crossing in a collection of curves, there is a relation, but the signs in that relation will depend on whether the crossing is between two distinct curves or the same curve, which is not local information.
%
\end{itemize}
In the undecorated case, the relations can be made local by sending a curve $c$ to $-\chi_c$.  This turns the last three rules into the \emph{Kauffman skein relation at q=-1} (\cite{Bul97},\cite{PS00}).  Such a fix will not be available in the decorated generality because of the orientation-dependance of the arcs.

A subsequent paper by the authors \cite{MSSkein} will explore two methods for fixing these short-comings:
\begin{itemize}
\item Cleverly choosing a sign-correction $w(c)$ for each curve $c$, so that the map $c\rightarrow (-1)^{w(c)}\chi_c$ has the desired properties.  
\item Twisting the definition of decorated $SL_2(\mathbb{C})$-local systems so that the corresponding character algebra $\widetilde{\Char}(\S,\M)$ can be canonically identified with a graphical algebra with the desired properties.
\end{itemize}
In both cases, the resulting graphical algebra will be the \emph{Kauffman skein algebra} at $q=1$ (where curves are allowed to have endpoints in $\M$).  The first approach is the decorated analog of Barrett's use of spin structures to flip signs in Kauffman skein algebras \cite{Bar99}, while the second corresponds to his observation that the $q=1$ skein algebra corresponded to certain flat $SL_2(\k)$-connections on the frame bundle.

%
%

%

\section{Invariants and concomitants on $End(\V)^m\oplus \V^n$}

The rest of the paper is devoted to the proofs of Theorem \ref{thm: main}.  The main ingredient for the proof of Theorem \ref{thm: main} will be presentations for the algebras of {invariants} and {matrix concomitants} on the space $End(\V)^m\oplus \V^n$ (Theorems \ref{thm: invpres} and \ref{thm: conpres}).  


\subsection{Preliminaries from invariant theory} This section collects the necessary results from invariant theory.  A reference is \cite{Stu08}.

 Let $G$ be a semisimple algebraic group over $\k$, and let $A$ be a $\k$-algebra with an action of $G$.  The subspace $A^G\subset A$ of $G$-invariant elements is a subalgebra, called the \textbf{invariant subalgebra}.  Taking invariants is functorial; an $SL_2(\k)$-equivariant morphism of algebras $f:A\rightarrow B$ restricts to a morphism $f:A^G\rightarrow B^G$.

The main goal of invariant theory is usually to find a presentation of $A^G$. A typical approach is to write $A = B / I$, where $B$ is an algebra with a $G$ action and $I$ is a $G$-stable two-sided ideal. Since $G$ is semisimple, it is a standard fact that 
\[A^G = (B/I)^G = B^G / I^G\]
However, finding generators for $I^G$ can be difficult without the help of the following lemma:
\begin{lemma}\label{lemma: invrel}
If $I$ is generated (as an ideal of $B$) by $G$-invariant elements $\{b_i\}$, then $I^G$ is the (two-sided ideal) in $B^G$ generated by $\{b_i\}$.
\end{lemma}

We also recall the important Reynolds operator:
\begin{lemma}[The Reynolds operator]\label{lemma: Reynolds} Let $G$ be a semisimple algebraic group over $\k$, and let $A$ be a $\k$-algebra with an action of $G$.  There is a $\k$-linear map $\gamma:A\dashrightarrow A^G$ such that
\begin{itemize}
\item $\gamma(a)=a$ if $a\in A^G\subset A$.
\item $\gamma(ab)=a\gamma(b)$ and $\gamma(ba)=\gamma(b)a$ if $a\in A^G\subset A$.
\item $\gamma$ is commutes with $G$-equivariant maps.
\end{itemize}
\end{lemma}

\subsection{Three isomorphisms} The following three linear maps define isomorphisms of $SL_2(\k)$-representations which will be used repeatedly.

\subsubsection{Transpose} First, let $\V^\vee$ denote the dual vector space to $\V$. The invariant form $\omega$ defines the \emph{transpose} map
$\perp: \V\stackrel{\sim}{\longrightarrow} \V^\vee$
by \[v^\perp:= \omega(v,-)\]

\subsubsection{Outer Product} Next, define the \emph{outer product} map
$ \Theta:\V\otimes \V\stackrel{\sim}{\longrightarrow} End(\V)$
by \[\Theta(v,w):=vw^\perp\] (note that endomorphisms of the form $vw^\perp$ span $End(\V)$).  It follows that
\begin{equation}
 tr(\Theta(v,w)) = w^\perp(v)=\omega(w,v)=-\omega(v,w)
\end{equation}

\subsubsection{Adjoint} Finally, define the \emph{adjoint} map $\iota:End(\V)\rightarrow End(\V)$ by
\[
 (vw^\perp)^\iota := -wv^\perp
\]
Since endomorphisms of the form $vw^\perp$ span $End(\V)$, this completely determines $\iota$.

\begin{prop}\label{prop: iotaprop}
The map $\iota$ has the following properties.
\begin{enumerate}
\item $\iota$ is an $SL_2(\k)$-equivariant anti-involution of the algebra $End(\V)$.
\item $\iota$ is the adjunction for the bilinear form $\omega$; i.e., $\omega(\mathsf{A}v,v')=\omega(v,\mathsf{A}^\iota v')$.
\item $\mathsf{A}+\mathsf{A}^\iota=tr(\mathsf{A})\cdot Id_\V$.  Therefore, $\mathsf{A}$ is scalar iff $\mathsf{A}$ is $\iota$-fixed.
\item $\mathsf{A}\mathsf{A}^\iota=\mathsf{A}^\iota \mathsf{A}=det(\mathsf{A})\cdot Id_\V$.  Therefore, $\mathsf{A}\in SL_2(\k)$ iff $\mathsf{A}\mathsf{A}^\iota=Id_\V$.
\item For $e_1,e_2$ an $\omega$-canonical basis for $\V$ (ie, $\omega(e_1,e_2)=1$), the action of $\iota$ on the corresponding matrices is
\begin{equation}
\gmat{a & b \\ c& d }^\iota = \gmat{d & -b \\ -c & a }
\end{equation}
\end{enumerate}
\end{prop}

\subsection{Mixed invariants and matrix concomitants}

Let
\begin{equation}
 \inv{m}{n}: = \mathcal{O}[End(\V)^m\oplus \V^n] = Sym_\k (End(\V)^m\oplus \V^n)^\vee
\end{equation}
denote the algebra of regular functions on the variety $End(\V)^m\oplus \V^n$.
The $SL_2(\k)$-action on $End(\V)$ and $\V$ extends to an action on $\inv{m}{n}$.

\begin{defn}
The $SL_2(\k)$-invariant subalgebra
$\inv{m}{n}^{SL_2}\subset \inv{m}{n}$ will be called the algebra of \textbf{mixed invariants}.
\end{defn}
The theory of mixed invariants and its presentations simultaneously generalizes the theory of invariant functions on $\V^n$ and the theory of invariant functions on $End(\V)^m$ (hence, `mixed').  As such, there are many partial results; some of the history of this problem will be reviewed in Section \ref{section: known}.


Now, denote by
\begin{equation}
 \con{m}{n}:= End(\V)\otimes \inv{m}{n}
\end{equation}
The multiplication in $End(\V)$ makes this into a non-commutative algebra over $\inv{m}{n}$, where $\inv{m}{n}\hookrightarrow \con{m}{n}$ as scalar matrices.  The group $SL_2(\k)$ acts diagonally on $\con{m}{n}$ by algebra automorphisms.  The algebra $\con{m}{n}$ is equivalent to the algebra of regular functions from $End(\V)^m\oplus \V^n$ to $End(\V)$; from this perspective, the $SL_2(\k)$-action is by 
post-conjugation.

\begin{defn}
The $SL_2(\k)$-invariant subalgebra
$\con{m}{n}^{SL_2}\subset \con{m}{n}$ will be called the algebra of \textbf{mixed matrix concomitants}.
\end{defn}

While an interesting object in its own right, the mixed matrix concomitants are most useful as an intermeditary in computing the mixed invariants and related algebras.  Specifically, some relations won't be $SL_2(\k)$-invariant in $\inv{m}{n}$, but will become invariant when extended to $\con{m}{n}$,\footnote{Or rather, the even part of $\con{m}{n}$; see Section \ref{section: gettinginvariantrelations}.} allowing Lemma \ref{lemma: invrel} to be used.

\begin{rem}(\emph{Notation for $\con{m}{n}$})
Plain math font $\{A,B,...\}$ will be used to denote generic elements in $\con{m}{n}$, sans serif font $\{\mathsf{A},\mathsf{B}...\}$ will be used to denote $End(\V)\subset \con{m}{n}$ (constant elements), and bold $\{\mathbf{A},\mathbf{B},...\}$ will be used to denote $SL_2$-invariant elements of $\con{m}{n}$.  This can be very useful for visually distinguishing between otherwise identical-looking results like Lemma \ref{lemma: kernel}, Corollary \ref{coro: kerspan} and Lemma \ref{lemma: invkernel}.
\end{rem}

\subsection{Maps between $\inv{m}{n}^{SL_2}$ and $\con{m}{n}^{SL_2}$} Results about algebras of mixed invariants and mixed matrix concomitants will be related by two maps, scalar inclusion and trace.

The scalar inclusion $\inv{m}{n}\hookrightarrow \con{m}{n}$ is $SL_2(\k)$-equivariant, and so it induces a scalar inclusion of invariants.
\[ -\cdot Id_\V:\inv{m}{n}^{SL_2}\hookrightarrow \con{m}{n}^{SL_2}\]
In this way, $\con{m}{n}^{SL_2}$ is an algebra over $\inv{m}{n}^{SL_2}$.


The anti-involution $\iota$ on $End(\V)$ extends to an anti-involution of $\con{m}{n}$, and all the analogous properties in Proposition \ref{prop: iotaprop} remain true.  In particular, for $A\in \con{m}{n}$, then  $A\in \inv{m}{n}\cdot Id_\V$ iff $A^\iota =A$, and so,
\begin{prop}
Under the scalar inclusion $\inv{m}{n}\hookrightarrow \con{m}{n}$, the algebra of mixed invariants $\inv{m}{n}^{SL_2}$ is the $\iota$-fixed subalgebra of $\con{m}{n}^{SL_2}$.
\end{prop}

The trace gives a linear map\footnote{Dashed arrows will be used to denote morphisms in a weaker category than their source and target; typically, linear maps between algebras which are not algebra maps.} $End(\V)\dashrightarrow \k$, which induces an $SL_2(\k)$-equivariant, $\inv{m}{n}$-module map
$tr:\con{m}{n}\dashrightarrow \inv{m}{n}$.  This then restricts to a $\inv{m}{n}^{SL_2}$-module map
\[ tr:\con{m}{n}^{SL_2}\dashrightarrow \inv{m}{n}^{SL_2}\]
%
Since $tr(f\cdot Id_\V)=2f$, the map $2^{-1}tr$ is a right inverse to scalar inclusion.  It follows that $tr$ surjects onto $\inv{m}{n}^{SL_2}$.  By Proposition \ref{prop: iotaprop},
\begin{equation}
A+A^\iota =tr(A)\cdot Id_\V
\end{equation}


\subsection{Elementary concomitants and invariants}

For $1\leq i\leq m$, let $\X_i\in \con{m}{n}$ denote the \emph{$i$th coordinate function},
\begin{equation}
\X_i(\mathsf{A}_1,\mathsf{A}_2,...,\mathsf{A}_m,v_1,v_2,...,v_m) := \mathsf{A}_i
\end{equation}
Since the action of $SL_2(\k)$ is by conjugation, $\X_i\in \con{m}{n}^{SL_2}$.

For $1\leq i,j\leq n$, let $\oute{i}{j}\in \con{m}{n}$ denote \emph{$(i,j)$th outer product function},
\begin{equation}
\oute{i}{j}(\mathsf{A}_1,\mathsf{A}_2,...,\mathsf{A}_m,v_1,v_2,...,v_m) := \Theta(v_i,v_j) = v_iv_j^\perp
\end{equation}
Since $\Theta$ is $SL_2(\k)$-equivariant, $\oute{i}{j}\in \con{m}{n}^{SL_2}$.

Because $\iota$ is $SL_2(\k)$-equivariant, if $\mathbf{A}\in \con{m}{n}^{SL_2}$, then $\mathbf{A}^\iota\in \con{m}{n}^{SL_2}$.  In particular, $\X_i^\iota$ and $\oute{i}{j}^\iota$ are also matrix concomitants, although $\oute{i}{j}^\iota=-\oute{j}{i}$, so only $\X_i^\iota$ provides a new example of a matrix concomitant.

Any of these concomitants, or more generally any word in these concomitants, can be made into an invariant by taking the trace.  

\subsection{Known results on $\inv{m}{n}^{SL_2}$ and $\con{m}{n}^{SL_2}$.}\label{section: known}

As a generalization of two well-known problems, there are many partial results on the structure of the algebras of mixed invariants and matrix concomitants.

The case $m=0$ is classical; see \cite{Wey39} or \cite{How89}.
\begin{thm}[Invariants on $\V^n$]
The algebra $\inv{0}{n}^{SL_2}$ is generated by $\{tr(\oute{i}{j})\}$, for $1\leq i,j\leq n$, with relations generated by $tr(\oute{i}{j})=-tr(\oute{j}{i})$ and
\[tr(\oute{i}{j})tr(\oute{i'}{j'}) = tr(\oute{i}{j'})tr(\oute{i'}{j})+tr(\oute{i}{i'})tr(\oute{j}{j'})\]
\end{thm}

The problem of finding invariants on $End(\V)^m$ was first proposed by Artin in \cite{Art69}.  A complete presentation of invariants and matrix concomitants on $End(\V)^m$ was found by Procesi
\footnote{Procesi worked in dimension $n$, but we restrict to $n=2$.}
 in \cite{Pro76}.  The invariants are generated by traces of strings of coordinate functions, and the relations are generated by a single class of relation.
\begin{thm}\cite[Theorems 1.3 and 4.5.a]{Pro76}\label{thm: invpro}
The algebra $\inv{m}{0}^{SL_2}$ is generated as a commutative algebra by $tr(\mathbf{A})$, as $\mathbf{A}$ runs over all words in the coordinate functions $\{\X_i\}$.  The relations are generated by \emph{Procesi's F-relation}:
\begin{align}\label{eq: F}
 tr(\mathbf{A}\mathbf{B}\mathbf{C})+tr(\mathbf{CBA})+tr(\mathbf{A})tr(\mathbf{B})tr(\mathbf{C}) &\\ =tr(\mathbf{B})tr(\mathbf{AC})&+tr(\mathbf{AB})tr(\mathbf{C})+ tr(\mathbf{A})tr(\mathbf{BC})
\end{align}
as $\mathbf{A}$, $\mathbf{B}$ and $\mathbf{C}$ run over all words in $\{\X_i\}$.
\end{thm}

The algebra of matrix concomitants $\con{m}{0}^{SL_2}$ can then be generated, over $\inv{m}{0}^{SL_2}$, by the coordinate functions, and the relations are again generated by a single class of relation.
\begin{thm}\cite[Theorems 2.1 and 4.5.b]{Pro76}\label{thm: conpro}
The algebra $\con{m}{0}^{SL_2}$ is generated, as an algebra over $\inv{m}{n}^{SL_2}$, by the coordinate functions $\{\X_i\}$.  The relations are generated by \emph{Procesi's G-relation}:
\begin{equation}
\mathbf{A}\mathbf{B}+\mathbf{B}\mathbf{A}-tr(\mathbf{A})\mathbf{B}-tr(\mathbf{B})\mathbf{A} -tr(\mathbf{AB})+tr(\mathbf{A})tr(\mathbf{B})=0
\end{equation}
as $\mathbf{A}$ and $\mathbf{B}$ run over all words in $\{\X_i\}$.
\end{thm}

%
%
%

Inspired by later results of Procesi, Brumfiel and Hilden produced a different presentation of $\con{m}{n}^{SL_2}$ (as a $\k$-algebra) which is specific to $SL_2(\k)$.
\begin{thm}\cite[Proposition 9.1.i]{BH95}\label{thm: conBH}
The algebra $\con{m}{0}^{SL_2}$ is generated by $\X_i$ and $\X_i^\iota$, with the relations generated by
\[ (\mathbf{A}+\mathbf{A}^\iota)\mathbf{B}=\mathbf{B}(\mathbf{A}+\mathbf{A}^\iota)\]
as $\mathbf{A}$ and $\mathbf{B}$ run over all words in $\{\X_i,\X_i^\iota\}$.
\end{thm}
As one would expect, $?^\iota$ denotes the anti-involution in the algebra generated by $\{\X_i,\X_i^\iota\}$ which interchanges $\X_i$ and $\X_i^\iota$.  Since $\mathbf{A}+\mathbf{A}^\iota=tr(\mathbf{A})\cdot Id_\V$, the theorem says the defining relation amongst the matrix concomitants is that the trace is central.

An interesting aspect of this presentation is that the algebra of matrix concomitants is produced directly, and then the algebra of invariants is found as the $\iota$-fixed subalgebra.

\section{Presentations of $\inv{m}{n}^{SL_2}$ and $\con{m}{n}^{SL_2}$}

This section provides presentations of the algebras of mixed invariants and mixed matrix concomitants which will be instrumental in proving Theorem \ref{thm: main}.  The problem of presenting these algebras in the mixed generality was first studied by Procesi \cite[Section 12]{Pro76} (for more general groups than $SL_2(\k)$), who found a generating set and outlined a brute force method for computing the relations.\footnote{However, as Procesi himself notes, producing a nice generating set for the relations via this method would likely be difficult or impossible.}  Our approach incorporates the $SL_2(\k)$-specific approach of Brumfiel and Hilden, and 
this specialization makes the problem tractable.


\subsection{The map $\nu$.}  Like the outer product map $\Theta: \V\otimes \V\rightarrow End(\V)$, there is an $SL_2(\k)$-equivariant isomorphism
\[ \V^n\otimes \V^n\stackrel{\sim}{\longrightarrow} End(\V^n)\simeq End(\V)^{n^2}\]
This induces a $SL_2(\k)$-equivariant map on coordinate rings
\begin{equation}
\nu:\inv{n^2}{0}\rightarrow \inv{0}{n}
\end{equation}
Here, $\nu(f)(v_1,...,v_n):= f(v_1v_1^\perp,v_1v_2^\perp,...v_nv_{n-1}^\perp,v_nv_n^\perp)$.  Tensoring this map with $\inv{m}{0}$
 or $\con{m}{0}$
  gives
\[ \nu:\inv{m+n^2}{0}\rightarrow \inv{m}{n}
,\;\;\;\;\;\;\;\;\;\; \nu:\con{m+n^2}{0}\rightarrow \con{m}{n}
\]
\begin{prop}\label{prop: invsurj}
The invariant maps $\nu:\inv{m+n^2}{0}^{SL_2}\rightarrow \inv{m}{n}^{SL_2}$ and $\nu:\con{m+n^2}{0}^{SL_2}\rightarrow \con{m}{n}^{SL_2}$ are surjective.
\end{prop}
\begin{proof}
In \cite[Theorem 12.1]{Pro76}, Procesi produces sets of generators for the algebra of mixed invariants, showing that they correspond to traces of products of matrices and evaluating $\omega$ on pairs of vectors.  Since these are in the image of $\nu$, surjectivity follows.  The analogous statement for mixed concomitants follows from the set of generators in \cite[Theorem 12.2.d]{Pro76}.
\end{proof}

\begin{rem}\label{rem: presentkernel}
Since the previous section provides a presentation of $\inv{m+n^2}{0}^{SL_2}$ and of $\con{m+n^2}{0}^{SL_2}$, this proposition reduces the problem of presenting the mixed invariants and matrix concomitants to the problem of understanding the kernel of $\nu$.
\end{rem}


\subsection{The kernel of $\nu$}

Under $\nu$, the first $m$ coordinate functions in $\con{m+n^2}{0}$ go to the coordinate functions in $\con{m}{n}$.  The last $n^2$ coordinate functions in $\con{m+n^2}{0}$ go to the outer product functions in $\con{m}{n}$, so by abuse of notation, the last $n^2$ coordinate functions in $\con{m+n^2}{0}$ will be denoted by $\oute{i}{j}$, for $1\leq i,j\leq n$.
\begin{lemma}\label{lemma: kernel}
The kernel of the map $\nu:\inv{n^2}{0}\rightarrow \inv{0}{n}$ is generated by
\begin{itemize}
\item $tr(\mathsf{A}\oute{i}{j})+tr(\mathsf{A}\oute{j}{i}^\iota)$, for $\mathsf{A}\in End(\V)$ and $1\leq i,j\leq n$
\item $tr(\mathsf{A}\oute{i}{j}\mathsf{B}\oute{i'}{j'})-tr(\mathsf{A}\oute{i}{j'})tr(\mathsf{B}\oute{i'}{j})$, for $\mathsf{A},\mathsf{B}\in End(\V)$ and $1\leq i,j\leq n$.
\end{itemize}
\end{lemma}
\begin{proof}
We use the isomorphism $\xi:\V^n\otimes \V^n\simeq (End(\V)^{n^2})^\vee$, given by
\[ \xi(v_p\otimes w_q):=v^\perp \oute{i}{j}w=\omega(v,\oute{i}{j}w)=tr(wv^\perp \oute{i}{j})\]
This induces an isomorphism of algebras
\[ \xi:Sym^\bullet(\V^n\otimes \V^n)\rightarrow \inv{n^2}{0}\]
The induced map
\[ \nu':Sym^\bullet(\V^n\otimes \V^n)\simeq \inv{n^2}{0}\stackrel{\nu}{\longrightarrow} \inv{0}{n}\simeq Sym^\bullet(\V^n)\]
is the natural symmetrization map.  These maps fit together into a commutative diagram,
\[\begin{array}{ccccc}
\mathcal{T}(\V^n\otimes \V^n) & \rightarrow & Sym^\bullet(\V^n\otimes \V^n) & \stackrel{\xi}{\longrightarrow} & \inv{n^2}{0} \\
\downarrow & & \downarrow \nu' & & \downarrow \nu \\
\mathcal{T}(\V^n) & \rightarrow & Sym^\bullet(\V^n) & \stackrel{\sim}{\longrightarrow}& \inv{0}{n}\\
\end{array}\]
Here, $\mathcal{T}$ denotes the tensor algebra and the maps from $\mathcal{T}$ to $Sym^\bullet$ are the symmetrization maps.

By definition, the kernel of
\[ \mathcal{T}(\V^n)\rightarrow Sym^\bullet(\V^n)\]
is spanned, as $j$ runs over all natural numbers, by elements of the form
\begin{equation}\label{eq: symkernel}
(v_1)_{p_1}\otimes (v_2)_{p_2}\otimes... \otimes (v_j)_{p_j} - (v_1)_{\sigma(p_1)}\otimes (v_2)_{\sigma(p_2)}\otimes... \otimes (v_j)_{\sigma(p_j)}
\end{equation}
for $v_i\in \V$, $1\leq p_i\leq j$ and $\sigma\in \Sigma_j$, the symmetric group on $j$ letters.  Since
$ \mathcal{T}(\V^n\otimes \V^n)\rightarrow \mathcal{T}(\V^n)$
is an inclusion, it follows that the kernel of
\[ \mathcal{T}(\V^n\otimes \V^n)\rightarrow Sym^\bullet(\V^n)\]
is spanned by elements of the form \eqref{eq: symkernel}.  Since $\mathcal{T}(\V^n\otimes \V^n)\rightarrow Sym^\bullet(\V^n\otimes \V^n)$ is surjective, the kernel of
\[ \nu':Sym^\bullet(\V^n\otimes \V^n)\rightarrow Sym^\bullet(\V^n)\]
is spanned by the the image of elements of the form \eqref{eq: symkernel}.

The symmetric group $\Sigma_j$ is generated by simple transpositions, and so the kernel of the map $\nu'$ is generated by two kinds of elements
\begin{itemize}
\item $v_i\otimes w_j-w_j\otimes v_i$, for $v,w\in \V$ and $1\leq i,j\leq n$.
\item $(v_i\otimes w_j)(v'_{i'}\otimes w'_{j'})-(v_i\otimes w'_{j'})(v'_{i'}\otimes w_{j})$, for $v,w,v',w'\in \V$ and $1\leq i,j,i',j'\leq n$.
\end{itemize}
We then compute the image of these generators under $\xi$.
\begin{eqnarray*}
\xi(v_i\otimes w_j-w_j\otimes v_i)
&=& tr(wv^\perp \oute{i}{j})-tr(vw^\perp \oute{j}{i})\\
&=& tr(wv^\perp \oute{i}{j})+ tr(wv^\perp \oute{j}{i}^\iota) \\
&=& tr(wv^\perp (\oute{i}{j}+\oute{j}{i}^\iota))
\end{eqnarray*}
Since products of the form $wv^\perp$ span $End(\V)$, this gives the first kind of generator.
\begin{eqnarray*}
\xi(v_i\otimes w_j)(v'_{i'}\otimes w'_{j'})
&=& tr(wv^\perp \oute{i}{j})tr(w'v'^\perp \oute{i'}{j'})\\
&=& tr(w (v'^\perp \oute{i'}{j'}w')v^\perp \oute{i}{j})\\
&=& tr(w'v^\perp \oute{i}{j}w v'^\perp \oute{i'}{j'})
\end{eqnarray*}
\begin{eqnarray*}
\xi(v_i\otimes w'_{j'})(v'_{i'}\otimes w_{j})
&=& tr(w'v^\perp \oute{i}{j'})tr(wv'^\perp \oute{i'}{j})
\end{eqnarray*}
Combining these gives the second class of relation.
\end{proof}

\begin{coro}\label{coro: kerspan}
The kernel $K$ of the map $\nu:\inv{m+n^2}{0}\rightarrow \inv{m}{n}$ is spanned by
\begin{itemize}
\item $tr(A\oute{i}{j})+tr(A\oute{j}{i}^\iota)$, for $A\in \con{m+n^2}{0}$ and $1\leq i,j\leq n$
\item $tr(A\oute{i}{j}B\oute{i'}{j'})-tr(A\oute{i}{j'})tr(B\oute{i'}{j})$, for $A,B\in \con{m+n^2}{0}$, and $1\leq i,j,i',j'\leq n$.
\end{itemize}
\end{coro}
\begin{proof}
The preceeding lemma provides a generating set for this kernel.  The kernel is then spanned by $\inv{m+n^2}{0}$-multiples of these generators.  In all cases, this coefficient may be pulled through the trace, and absorbed into the definition of $A$ or $B$.
\end{proof}

\subsection{The invariant kernel of $\nu$} The next step is to find the kernel of the map $\nu$ when restricted to the invariants.

\begin{lemma}\label{lemma: invkernel}
The kernel of the map $\nu:\inv{m+n^2}{0}^{SL_2}\rightarrow \inv{m}{n}^{SL_2}$ is spanned  by
\begin{itemize}
\item $tr(\mathbf{A}\oute{i}{j})+tr(\mathbf{A}\oute{j}{i}^\iota)$, for $\mathbf{A}\in \con{m+n^2}{0}^{SL_2}$ and $1\leq i,j\leq n$
\item $tr(\mathbf{A}\oute{i}{j}\mathbf{B}\oute{i'}{j'})-tr(\mathbf{A}\oute{i}{j'})tr(\mathbf{B}\oute{i'}{j})$, for $\mathbf{A},\mathbf{B}\in \con{m+n^2}{0}^{SL_2}$ and $1\leq i,j\leq n$.
\end{itemize}
\end{lemma}
\begin{proof}
Let $R_{i,j}$ denote the $SL_2$-subrepresentation of $\inv{m+n^2}{0}^{SL_2}$
spanned by
\begin{equation}\label{eq: rel1}
tr(A\oute{i}{j})+tr(A\oute{j}{i}^\iota),
\end{equation}
for $A\in \con{m+n^2}{0}$. There is a $SL_2$-equivariant surjection $\con{m+n^2}{0}\rightarrow R_{i,j}$ which sends $A$ to \eqref{eq: rel1}.  The induced map on invariants is a surjection, and so $R_{i,j}^{SL_2}$ is spanned by
\begin{equation}\label{eq: invrel1}
tr(\mathbf{A}\oute{i}{j})+tr(\mathbf{A}\oute{j}{i}^\iota),
\end{equation}
for $\mathbf{A}\in \con{m+n^2}{0}^{SL_2}$.

Let $R_{i,j,i',j'}$ denote the $SL_2$-subrepresentation $\inv{m+n^2}{0}^{SL_2}$ spanned by
\begin{equation}\label{eq: rel2}
tr(A\oute{i}{j}B\oute{i'}{j'})-tr(A\oute{i}{j'})tr(B\oute{i'}{j}),
\end{equation}
for $A,B\in \con{m+n^2}{0}$. There is a $SL_2$-equivariant surjection
\begin{equation}\label{eq: maprel2}
\con{m+n^2}{0}\otimes_{\inv{m+n^2}{0}}\con{m+n^2}{0}\rightarrow R_{i,j,i',j'}\end{equation}
which sends $A\otimes B$ to \eqref{eq: rel2}.
The induced map on invariants is a surjection.

Choose a basis $v_1,v_2$ for $\V$, and let $e_{ij}$ denote the $(i,j)$ elementary matrix in this basis.
\begin{lemma}
The invariants $(\con{m+n^2}{0}\otimes_{\inv{m+n^2}{0}}\con{m+n^2}{0})^{SL_2}$ are spanned by $\mathbf{A}\otimes \mathbf{B}$ and $\sum_{i,j\in\{1,2\}} e_{ij}\mathbf{A} \otimes e_{ji}\mathbf{B}$.
\end{lemma}
\begin{proof} Consider $\inv{m+n^2+2}{0}$, where the two final coordinate functions are denoted $\mathbf{Y}_1$ and $\mathbf{Y}_2$.  Then there is an $SL_2$-equivariant inclusion
\begin{equation}\label{eq: doubleinclusion}
 \con{m+n^2}{0}\otimes_{\inv{m+n^2}{0}}\con{m+n^2}{0}\hookrightarrow \inv{m+n^2+2}{0}\end{equation}
which sends $A\otimes B$ to $tr(A\mathbf{Y}_1B\mathbf{Y}_2)$.  The image of this map consists of the elements of $\inv{m+n^2+2}{0}$ which are linear in $\mathbf{Y}_1$ and in $\mathbf{Y}_2$.

By Theorem \ref{thm: invpro}, the subspace of $\inv{m+n^2+2}{0}^{SL_2}$ which is linear in $\mathbf{Y}_1$ and in $\mathbf{Y}_2$ is spanned by $tr(\mathbf{A}\mathbf{Y}_1\mathbf{B}\mathbf{Y}_2)$ and $tr(\mathbf{A}\mathbf{Y}_1)tr(\mathbf{B}\mathbf{Y}_2)$, for $\mathbf{A},\mathbf{B}\in \con{m+n^2+2}{0}^{SL_2}$.  Note that
\[\sum_{i,j\in\{1,2\}}tr(e_{ij}\mathbf{A}\mathbf{Y}_1e_{ji}\mathbf{B}\mathbf{Y}_2) =
\sum_{i,j\in\{1,2\}}(\mathbf{A}\mathbf{Y}_1)_{jj}(\mathbf{B}\mathbf{Y}_2)_{ii} = tr(\mathbf{A}\mathbf{Y}_1)tr(\mathbf{B}\mathbf{Y}_2)
\]
It follows that the image of the span of
\[\{\mathbf{A}\otimes\mathbf{B}\}\bigcup\left\{\sum_{i,j\in \{1,2\}}e_{ij}\mathbf{A}\otimes e_{ji}\mathbf{B}\right\}, \text{ for }{\mathbf{A},\mathbf{B}\in \con{m+n^2+2}{0}^{SL_2}}\] under the map \eqref{eq: doubleinclusion} spans $\inv{m+n^2+2}{0}^{SL_2}$.
\end{proof}

The map \eqref{eq: maprel2} sends $\mathbf{A}\otimes \mathbf{B}$ to
\begin{equation}\label{eq: invrel2}
tr(\mathbf{A}\oute{i}{j}\mathbf{B}\oute{i'}{j'})-tr(\mathbf{A}\oute{i}{j'})tr(\mathbf{B}\oute{i'}{j}),
\end{equation}
and $\sum_{i,j\in\{1,2\}} e_{ij}\mathbf{A} \otimes e_{ji}\mathbf{B}$ to
\begin{align*}
\sum_{i,j\in\{1,2\}} &tr(e_{ij}\mathbf{A} \oute{i}{j} e_{ji}\mathbf{B}\oute{i'}{j'}) -  tr(e_{ij}\mathbf{A} \oute{i}{j})tr( e_{ji}\mathbf{B}\oute{i'}{j'}) \\
=& \sum_{i,j\in\{1,2\}} (\mathbf{A} \oute{i}{j})_{jj}(\mathbf{B}\oute{i'}{j'})_{ii} -  (\mathbf{A} \oute{i}{j})_{ji}(\mathbf{B}\oute{i'}{j'})_{ij}\\
=&tr(\mathbf{A} \oute{i}{j})tr(\mathbf{B}\oute{i'}{j'})-  tr(\mathbf{A} \oute{i}{j}\mathbf{B}\oute{i'}{j'})
\end{align*}
Notice that this is of the same form as \eqref{eq: invrel2}, except with a minus sign and $i'$ and $j$ exchanged.  It follows that $R_{i,j,i',j'}^{SL_2}$ is spanned by elements of these two forms.

Corollary \ref{coro: kerspan} states that the kernel $K$ of $\nu$ is
\[ \sum_{1\leq i,j\leq n}R_{i,j}+\sum_{1\leq i,j,i',j'\leq n}R_{i,j,i',j'}\]
The Reynolds operator (Lemma \ref{lemma: Reynolds}) is a linear projection $\gamma :\inv{m+n^2}{0}\dashrightarrow \inv{m+n^2}{0}^{SL_2}$, which sends subrepresentations to their invariant subspaces.  Applying this to $K$ gives
\[\gamma\left(\sum_{1\leq i,j\leq n}R_{i,j}+\sum_{1\leq i,j,i',j'\leq n}R_{i,j,i',j'}\right) =
\sum_{1\leq i,j\leq n}R_{i,j}^{SL_2}+\sum_{1\leq i,j,i',j'\leq n}R_{i,j,i',j'}^{SL_2}\]
Therefore, $K^{SL_2}$ is spanned by elements of the form \eqref{eq: invrel1} and \eqref{eq: invrel2}, as the $\mathbf{A}$ and $\mathbf{B}$ run over $\con{m+n^2}{0}$ and $i,j,i',j'$ run over $1,...,n$.
%
%
\end{proof}

\begin{lemma}\label{lemma: conkernel}
The kernel of the map $\nu:\con{m+n^2}{0}^{SL_2}\rightarrow \con{m}{n}^{SL_2}$ is generated by
\begin{itemize}
\item $\oute{i}{j}+\oute{j}{i}^\iota$, for $1\leq i,j\leq n$
\item $\oute{i}{j}\mathbf{B}\oute{i'}{j'}-tr(\mathbf{B}\oute{i'}{j})\oute{i}{j'}$ or equivalently,
\[\oute{i}{j}\mathbf{B}\oute{i'}{j'}-\mathbf{B}\oute{i'}{j}\oute{i}{j'}-\oute{i'}{j}^\iota\mathbf{B}^\iota\oute{i}{j'}\] for $\mathbf{B}\in \con{m+n^2}{0}^{SL_2}$ and $1\leq i,j,i',j'\leq n$.
\end{itemize}
\end{lemma}
\begin{proof}
We consider $\inv{m+1}{n}$ with final coordinate function denoted $\mathbf{Y}$.  There is an inclusion
$\con{m}{n}\rightarrow \inv{m+1}{n}$, which sends $A$ to $tr(A\mathbf{Y})$.  The image is the subspace of $\inv{m+1}{n}$ which is linear in $\mathbf{Y}$.  We have a commutative diagram
\[\begin{array}{ccc}
\con{m+n^2}{0}^{SL_2} & \hookrightarrow & \inv{m+n^2+1}{0}^{SL_2} \\
\downarrow \nu & & \downarrow \nu \\
\con{m}{n}^{SL_2}  &\hookrightarrow & \inv{m+1}{n}^{SL_2}\\
\end{array}\]
Hence, the kernel of $\nu$ in $\con{m+n^2}{0}^{SL_2}$ is the preimage of the kernel of $\nu$ in $\inv{m+n^2+1}{0}^{SL_2}$.

If $tr(\mathbf{A}\oute{i}{j})+tr(\mathbf{A}\oute{j}{i}^\iota)\in \con{m+n^2+1}{0}$ is linear in $\mathbf{Y}$, then $\mathbf{A}=\mathbf{B}\mathbf{Y}\mathbf{C}$, for some $\mathbf{B},\mathbf{C}\in \con{m+n^2}{0}$.  Then
\[ \mathbf{C}(\oute{i}{j}+\oute{j}{i}^\iota)\mathbf{B}\stackrel{tr(-\mathbf{Y})}{\longrightarrow} tr(\mathbf{C}(\oute{i}{j}+\oute{j}{i}^\iota)\mathbf{B}\mathbf{Y})=tr(\mathbf{A}\oute{i}{j})+tr(\mathbf{A}\oute{j}{i}^\iota)\]
Thus, the preimage of $tr(\mathbf{A}\oute{i}{j})+tr(\mathbf{A}\oute{j}{i}^\iota)$ is in the ideal generated by $\oute{i}{j}+\oute{j}{i}^\iota$.  A similar argument works for the other class of relations.
\end{proof}

\subsection{Presentations of mixed invariants and mixed concomitants}

By Proposition \ref{prop: invsurj}, a presentation of $\inv{m}{n}^{SL_2}$ (resp. $\con{m}{n}^{SL_2}$) can be produced by choosing a presentation of $\inv{m+n^2}{0}^{SL_2}$ (resp. $\con{m+n^2}{0}^{SL_2}$) and adding a relation coming from the generators for the kernel of $\nu$.

Combining Procesi's presentation of $\inv{m+n^2}{0}^{SL_2}$ with the previous presentation of the kernel of $\nu$ gives the following.

\begin{thm}\label{thm: invpres}
The algebra $\inv{m}{n}^{SL_2}$ is generated as a commutative algebra by $tr(\mathbf{A})$, as $\mathbf{A}$ runs over all words in the coordinate functions $\{\X_i,\oute{i}{j}\}$.
The relations are generated by
\begin{itemize}
\item Procesi's F-relation \eqref{eq: F},
\item
$tr(\mathbf{A}\oute{i}{j})=tr(\mathbf{A}\oute{j}{i})-tr(\mathbf{A})tr(\oute{j}{i})$, for $\mathbf{A}$ a word in the $\{\X_i,\oute{i}{j}\}$ and $1\leq i,j\leq n$,
\item $tr(\mathbf{A}\oute{i}{j}\mathbf{B}\oute{i'}{j'})=tr(\mathbf{A}\oute{i}{j'})tr(\mathbf{B}\oute{i'}{j})$, for $\mathbf{A},\mathbf{B}$ words in $\{\X_i,\oute{i}{j}\}$ and $1\leq i,j,i',j'\leq n$.
\end{itemize}
\end{thm}
\begin{proof}
The generators and the first relation come from Theorem \ref{thm: invpro}.  The second two relations come from Lemma \ref{lemma: invkernel}.  Note that any occurence of $\X_i^\iota$ or $\oute{i}{j}^\iota$ may be replaced by $tr(\X_i)-\X_i$ or $tr(\oute{i}{j})-\oute{i}{j}$, so as to avoid using $\iota$. In this way, the second relation is a reformulation of $tr(\mathbf{A}\oute{i}{j})=-tr(\mathbf{A}\oute{j}{i}^\iota)$.
\end{proof}

Combining Brumfiel and Hilden's presentation of $\con{m+n^2}{0}^{SL_2}$ with the previous presentation of the kernel of $\nu$ gives the following.

\begin{thm}\label{thm: conpres}
The algebra $\con{m}{n}^{SL_2}$ is generated by
\begin{itemize}
\item \emph{coordinate functions} $\X_i$, for $1\leq i\leq m$,
\item \emph{adjoint coordinate functions} $\X_i^\iota$, for $1\leq i\leq m$,
\item \emph{outer product functions} $\oute{i}{j}$, for $1\leq i,j\leq n$.
\end{itemize}
The relations may be written two ways.  Let $?^\iota$ be the anti-involution with $(\X_i)^\iota:=\X_i^\iota$, and $(\oute{i}{j})^\iota:=-\oute{j}{i}$.  Then the relations are generated by
\begin{enumerate}
\item $(\mathbf{A}+\mathbf{A}^\iota)\mathbf{B}=\mathbf{B}(\mathbf{A}+\mathbf{A}^\iota)$, for all $\mathbf{A},\mathbf{B}$ words in $\{\X_i,\X_i^\iota,\oute{i}{j}\}$.
\item $\oute{i}{j}\mathbf{A}\oute{i'}{j'} = \mathbf{A}\oute{i'}{j}\oute{i}{j'}-\oute{j}{i'}\mathbf{A}^\iota\oute{i}{j'}$,
 for all $1\leq i,j,i',j'\leq n$ and $\mathbf{A}$ a word in $\{\X_i,\X_i^\iota,\oute{i}{j}\}$.
\end{enumerate}

Let $tr(\mathbf{A}):=\mathbf{A}+\mathbf{A}^\iota$.  Then the relations are generated by
\begin{enumerate}
\item $ tr(\mathbf{A})\mathbf{B}=\mathbf{B}tr(\mathbf{A})$, for all $\mathbf{A},\mathbf{B}$ words in $\{\X_i,\X_i^\iota,\oute{i}{j}\}$.
\item $\oute{i}{j}\mathbf{A}\oute{i'}{j'} = tr(\mathbf{A}\oute{i'}{j})\oute{i}{j'}$ for all $1\leq i,j,i',j'\leq n$ and $\mathbf{A}$ a word in $\{\X_i,\X_i^\iota,\oute{i}{j}\}$.
\end{enumerate}
\end{thm}
\begin{proof}
Theorem \ref{thm: conBH} implies that $\con{m}{n}^{SL_2}$ is generated by $\{\X_i,\X_i^\iota,\oute{i}{j},\oute{i}{j}^\iota\}$, and that the trace of any element is central.  Lemma \ref{lemma: conkernel} shows the rest of the relations are given by $\oute{i}{j}=-\oute{j}{i}^\iota$ and
$\oute{i}{j}\mathbf{A}\oute{i'}{j'}=tr(\mathbf{A}\oute{i'}{j})\oute{i}{j'}$.
The relation $\oute{i}{j}=-\oute{j}{i}^\iota$ may be used to eliminate the variables $\oute{i}{j}^\iota$, which in turn eliminates the need for that relation.  The remaining generators and relations give the theorem as stated.
\end{proof}

\section{Presenting the character algebra}  

In this section we apply the results of the previous section to the decorated character algebra.

\subsection{A presentation of the representation algebra}\label{section: repalg}

  First, we will need a presentation of the representation algebra $\Rep(G,M)$; recall that $\Rep(G,M)$ is the coordinate ring of the representation scheme $Hom((G,M),(SL_2(\k),\V))$.

(\emph{Fixing a presentation for $(G,M)$}) The definition of the representation algebra will require an explicit presentation of $(G,M)$.  Choose a finite generating set $\{g_i\}_{1\leq i\leq n}\subset G$ for $G$, and a finite set of $G$-orbit representatives $\{p_j\}_{1\leq j\leq n}\subset M$.\footnote{For simplicity, we will assume the generating set $\{g_i\}$ is closed under inverses.  That way, every element of $g$ can be written as a positive word in $\{g_i\}$.}  Let $\{g_{i_1}g_{i_2}...g_{i_k}\}$ be a generating set of relations for $G$, and for each $j$, let $\{g_{j_1}g_{j_2}...g_{j_k}\}$ be a generating set for the stablizer of $p_j$ in $M$.

\begin{defn} \label{defn: repalg} Define the \textbf{representation algebra} $\Rep(G,M)$ to be the quotient of $\inv{m}{n}$ by the ideal generated by...
\begin{enumerate}
\item $\det(\X_i)-1$ for all $1\leq i\leq n$,
\item $\varphi\left(\X_{i_1}\X_{i_2}...\X_{i_j}-Id_\V\right)$ for each relation in $G$, and each $\varphi\in End(\V)^\vee$.
\item $\varphi\left((\X_{j_1}\X_{j_2}...\X_{j_k}-Id_\V)\mathbf{v}_{j}\right)$ for each $1\leq j\leq m$, each generator of the stablizer of $p_j$, and each $\varphi\in \V^\vee$.
\end{enumerate}
\end{defn}
\begin{rem}
If $M$ is a free $G$-set, then the third class of relations is redundant.  If $G$ is a free group and $M$ is a free $G$-set, then the second and third classes of relations are empty.
\end{rem}


The representation is the universal algebra of decorated representations of $(G,M)$ into $(SL_2(\k),\V)$ in the following sense.

\begin{prop}\label{prop: unirep}
Let $A$ be a commutative $\k$-algebra.  Then there is a natural, $SL_2$-equivariant bijection (functorial in $A$)
\[ \left\{ \begin{array}{c}\text{group action maps}\\ (G,M)\rightarrow (SL_2(A),A\otimes \V)\end{array}\right\}
\stackrel{\sim}{\longrightarrow}
\left\{ \begin{array}{c}\text{$\k$-algebra maps}\\ \Rep(G,M)\rightarrow A\end{array}\right\}\]
\end{prop}
\begin{proof}
Given a map of group actions
\[ f=(f_G,f_M): (G,M)\rightarrow (SL_2(A),A\otimes \V),\]
define a map $\alpha_f:\inv{m}{n}\rightarrow A$ by 
\[ \alpha_f(\varphi(\X_i))=f_G(g_i), \text{ for each } 1\leq i\leq m\text{, and each }\varphi\in End(\V)^\vee\]
\[ \alpha_f(\varphi(\mathbf{v}_i)) =f_M(v_j), \text{ for each } 1\leq j\leq n\text{, and each }\varphi\in \V^\vee\]
Since $\det(f_G(g_i))=1$, the map $\alpha_f$ kills the elements $\det(\X_i)-1$ in $\inv{m}{n}$.  Since $f$ is a map of group actions, $\alpha_f$ kills the remaining relations for $\Rep(G,M)$, and so $\alpha_f$ descends to a well-defined map
$ \Rep(G,M)\rightarrow A$.

Given a map $\alpha:\Rep(G,M)\rightarrow A$, define a map of group actions $f_\alpha=(f_G,f_M)$ as follows.  Let $f_G(g_i)$ be the unique element in $End(\V)$ such that
\[ \varphi(f_G(g_i))=\alpha(\varphi(\X_i)), \text{ for each }\varphi\in End(\V)^\vee\]
Let $f_M(v_j)$ be the unique element in $\V$ such that
\[ \varphi(f_M(v_j))=\alpha(\varphi(\mathbf{v}_j)), \text{ for each }\varphi\in \V^\vee\]
Since $\alpha$ kills $\det(\X_i)-1$, then $f_G(g_i)\in SL_2(\k)$.  The remaining relations in $\Rep(G,M)$ imply that $f_\alpha:(G,M)\rightarrow (SL_2(\k),\V)$ is a map of group actions.

These two constructions are directly seen to be mutual inverses which are $SL_2$-equivariant.  Functoriality is straight-forward.
\end{proof}

As a consequence of the above universal property, the algebra $\Rep(G,M)$ is independent of the choice of presentation for $(G,M)$.

\begin{coro}
The $\k$-points of the scheme $Rep(G,M):=Spec(\Rep(G,M))$ are in bijection with $Hom((G,M),(SL_2(\k),\V))$.
\end{coro}

\subsection{Getting invariant relations}\label{section: gettinginvariantrelations}

The aim is to present the $SL_2$-invariants in $\Rep(G,M)$.  The previous section gives a presentation of $\inv{m}{n}^{SL_2}$.  However, the second two relations given in Definition \ref{defn: repalg} are not $SL_2$-equivariant, and so Lemma \ref{lemma: invrel} cannot be used.  This is fixed by tensoring with $End(\V)$, and passing to the even subalgebra.

The defining surjection
\[ \pi:\inv{m}{n}\rightarrow \Rep(G,M)\]
gives a surjection
\[ \pi_\mathcal{E}:\con{m}{n}\rightarrow End(\V)\otimes \Rep(G,M)\]

Using $-Id_\V\in SL_2(\k)$, any $SL_2(\k)$-representation $V$ splits into \textbf{even} and \textbf{odd} summands,
\[ V^e:=\{v\in V| (-Id_\V )\cdot v=v\}\]
\[ V^o:=\{v\in V| (-Id_\V )\cdot v=-v\}\]
The even part contains the $SL_2(\k)$-invariants; $V^{SL_2}\subseteq V^e$.

One checks that $End(\V)$ is even and $\V$ is odd.  It then follows that $\inv{m}{n}^e\subset \inv{m}{n}$ is the subalgebra of functions which have even degree in $\V^n$,
\[ \inv{m}{n}^e=\{f\in \inv{m}{n} \,|\, f(\mathsf{A}_1,...,\mathsf{A}_n,v_1,...,v_m)=f(\mathsf{A}_1,...,\mathsf{A}_n,-v_1,,...,-v_m)\}\]
The even parts of $\con{m}{n}$ and $\Rep(G,M)$ are similarly equal to the subalgebras of functions even in the $\V^n$ part.

Let $\pi_\mathcal{E}^e$ denote the restricted surjection
\[ \pi_\mathcal{E}^e:\con{m}{n}^e\rightarrow End(\V)\otimes   \Rep(G,M)^e\]
\begin{rem}\emph{Technical necessities.}\label{rem: gettinginvariant}  Tensoring with $End(\V)$ and restricting to the even part are necessary so that the kernel of $\pi^e_\mathcal{E}$ is generated by $SL_2$-invariants.  However, in simple cases this is unnecessary.  If $M$ is a free $G$-set, $\pi_\mathcal{E}$ already has an invariantly-generated kernel.  If $G$ is a free group and $M$ is a free $G$-set, $\pi$ already has an invariantly-generated kernel.
\end{rem}

\begin{lemma}
The kernel of $\pi^e_{\mathcal{E}}$ is generated by
\begin{enumerate}
\item $\X_i\X_i^\iota-Id_\V$ for all $1\leq i\leq n$,
\item $\X_{i_1}\X_{i_2}...\X_{i_j}-Id_\V$ for each relation in $G$.
\item $(\X_{j_1}\X_{j_2}...\X_{j_k}-Id_\V)\oute{j}{j'}$ for each $1\leq j,j'\leq m$, and each generator of the stablizer of $p_j$.
\end{enumerate}
\end{lemma}
\begin{proof}Let $I$ denote the two-sided ideal in $\con{m}{n}^e$ generated by the elements in the statement of the lemma, so that we wish to prove $I=\ker(\pi^e_{\mathcal{E}})$.

The kernel of $\pi_\mathcal{E}$ is spanned by elements of the form $Ar$, where $A\in \con{m}{n}$ and $r$ runs over all the relations in Definition \ref{defn: repalg}.  Such an element is in $\con{m}{n}^e$ if both $A$ and $r$ are in $\con{m}{n}^e$, or if they are both  in $\con{m}{n}^o$.  The first two relations in Definition \ref{defn: repalg} are even, the third is odd.  In the third case, we can find $\{\varphi_i\}\in \V^\vee$ and $A_i\in\, \con{m}{n}^e$ such that
\[ A\varphi\left((\X_{j_1}\X_{j_2}...\X_{j_k}-Id_\V)\mathbf{v}_{j}\right) = \left[\sum_{1\leq i\leq m} A_i\varphi_i(\mathbf{v}_i)\right]\varphi\left((\X_{j_1}\X_{j_2}...\X_{j_k}-Id_\V)\mathbf{v}_{j}\right)\]
For any $\varphi,\varphi'\in \V^\vee$, there is a $\varphi''\in End(\V)^\vee$ such that
\[ \varphi'(\mathbf{v}_i)\varphi\left((\X_{j_1}\X_{j_2}...\X_{j_k}-Id_\V)\mathbf{v}_{j}\right)=\varphi''\left((\X_{j_1}\X_{j_2}...\X_{j_k}-Id_\V)\oute{j}{i}\right)\]

The kernel of $\pi^e_{\mathcal{E}}$ is then spanned by elements of the form (as $A$ runs over $^e\con{m}{n}$):
\begin{enumerate}
\item $A(\det(\X_i)-1)$.
\item $A\varphi\left(\X_{i_1}\X_{i_2}...\X_{i_j}-Id_\V\right)$ for each relation in $G$, and each $\varphi\in End(\V)^\vee$.
\item $A\varphi\left((\X_{j_1}\X_{j_2}...\X_{j_k}-Id_\V)\oute{j}{i}\right)$ for each $1\leq i,j\leq m$, each generator of the stablizer of $p_j$, and each $\varphi\in End(\V)^\vee$.
\end{enumerate}

By Proposition \ref{prop: iotaprop}, $\X_i\X_i^\iota-Id_\V=(\det(\X_i)-1) Id_\V$.  Then $I\subseteq \ker(\pi^e_\mathcal{E})$, since
\[Id_\V\in End(\V)\otimes End(\V)^\vee \]

Now, let $A\in \con{m}{n}^e$ and $\varphi\in End(\V)^\vee$.  Assume first that $A=fvw^\perp$ and $\varphi=tr(v'w'^\perp-)$, for $f\in \inv{m}{n}^e$ and $v,w,v',w'\in \V$.  Then, for $B\in I$,
\[ A\varphi(B) = fvw^\perp tr(v'w'^\perp B) = f(vw'^\perp) B(v'w^\perp) \in I\]
By linearity, $A\varphi(B)\in I$ for general $A$ and $\varphi$.  Then the above spanning set for $\ker(\pi^e_{\mathcal{E}})$is contained in $I$.  Therefore, $I=\ker(\pi^e_{\mathcal{E}})$.
\end{proof}

The map $\pi_\mathcal{E}^e$ then has a kernel which is generated by $SL_2$-invariants.  By Lemma \ref{lemma: invrel}, the kernel of this map when restricted to invariants is generated by the same set.

\begin{coro}\label{coro: matcharpres}
The kernel of $\pi_{\mathcal{E}}^{SL_2}:\con{m}{n}^{SL_2}\rightarrow (End(\V)\otimes \Rep(G,M))^{SL_2}$ is generated by
\begin{enumerate}
\item $\X_i\X_i^\iota-Id_\V$ for all $1\leq i\leq n$,
\item $\X_{i_1}\X_{i_2}...\X_{i_j}-Id_\V$ for each relation in $G$.
\item $(\X_{j_1}\X_{j_2}...\X_{j_k}-Id_\V)\oute{j}{j'}$ for each $1\leq j,j'\leq m$, and each generator of the stablizer of $p_j$.
\end{enumerate}
\end{coro}

This corollary can be combined with the presentation of $\con{m}{n}^{SL_2}$ (Theorem \ref{thm: conpres}) to yield a 
presentation of $(End(\V)\otimes \Rep(G,M))^{SL_2}$.
\begin{coro}
The algebra $(End(\V)\otimes \Rep(G,M))^{SL_2}$ is generated by $\{\X_i,\X_i^\iota,\oute{j}{k}\}$ for $1\leq i\leq m$ and $1\leq j,k\leq m$, with relations generated by
\begin{enumerate}
 \item $\X_i\X_i^\iota-Id_\V$ for all $1\leq i\leq n$,
\item $\X_{i_1}\X_{i_2}...\X_{i_j}-Id_\V$ for each relation in $G$.
\item $(\X_{j_1}\X_{j_2}...\X_{j_k}-Id_\V)\oute{j}{j'}$ for each $1\leq j,j'\leq m$, and each generator of the stablizer of $p_j$.
\item $(\mathbf{A}+\mathbf{A}^\iota)\mathbf{B}=\mathbf{B}(\mathbf{A}+\mathbf{A}^\iota)$, for all $\mathbf{A},\mathbf{B}$ words in $\{\X_i,\X_i^\iota,\oute{i}{j}\}$.
\item $\oute{i}{j}\mathbf{A}\oute{i'}{j'} = \mathbf{A}\oute{i'}{j}\oute{i}{j'}-\oute{j}{i'}\mathbf{A}^\iota\oute{i}{j'}$,
 for all $1\leq i,j,i',j'\leq n$ and $\mathbf{A}$ a word in $\{\X_i,\X_i^\iota,\oute{i}{j}\}$.
\end{enumerate}
\end{coro}

\subsection{A presentation of the matrix character algebra}
There is a more natural presentation of the algebra $(End(\V)\otimes \Rep(G,M))^{SL_2}$ which is independent of the choice of presentation of $(G,M)$. Let $\k[G]$ be the group ring of $G$, and let
\[\mathcal{T}_GM^2:=\mathcal{T}_{\k[G]}(\k M\otimes \k M)\]
be the tensor algebra of $\k M\otimes \k M$ over $\k[G]$.  For $g\in G$, the corresponding element in $\k[G]$ will be denoted $\X_g$; for $p,q\in M$, the corresponding element in $\k M\otimes \k M$ will be denoted $\oute{p}{q}$.
\begin{lemma} The tensor algebra $\mathcal{T}_GM^2$ is naturally isomorphic to
the abstract algebra generated by $\{\X_i,\X_i^\iota,\oute{j}{k}\}$ for $1\leq i\leq m$ and $1\leq j,k\leq m$, with relations generated by
\begin{enumerate}
\item $\X_i\X_i^\iota-Id_\V$ for all $1\leq i\leq n$,
\item $\X_{i_1}\X_{i_2}...\X_{i_j}-Id_\V$ for each relation in $G$.
\item $(\X_{j_1}\X_{j_2}...\X_{j_k}-Id_\V)\oute{j}{j'}$ for each $1\leq j,j'\leq m$, and each generator of the stablizer of $p_j$.
\end{enumerate}
\end{lemma}
\begin{proof}
First, note that Relation (1) in Corollary \ref{coro: matcharpres} means that $\X_i^\iota=\X_i^{-1}$.
For $g\in G$, let $g_{k_1}g_{k_2}...g_{k_l}=g$ be a word for $g$ in the generating set $\{g_i,g_i^{-1}\}$. Then
\[ \X_g= \X_{k_1}\X_{k_2}...\X_{k_l}\]
Relation (2) in Corollary \ref{coro: matcharpres} guarentees this element is independent of the choice of word for $g$.  Similarly, for $p,q\in M$, let $g_{i_1}g_{i_2}...g_{i_k}p_i=p$ and $g_{j_1}g_{j_2}...g_{j_{l}}p_{j}=q$ be words for $p$ and $q$.  Then
\[ \oute{p}{q} =\X_{{i_1}}\X_{{i_2}}...\X_{{i_k}}\oute{i}{j}\X_{{j_l}}^\iota...\X_{{j_2}}^\iota\X_{{j_{l}}}^\iota\]
Relation (3) in Corollary \ref{coro: matcharpres} guarentees this element is independent of the choice of words for $p$ and $q$.
\end{proof}

There is then a surjection
\begin{equation}\label{eq: tensorsurj}
\mathcal{T}_GM^2\rightarrow (End(\V)\otimes \Rep(G,M))^{SL_2}
\end{equation}
whose kernel is generated by relations of type (4) and (5) in Corollary \ref{coro: matcharpres}.  The elements $\X_g$ and $\oute{p}{q}$ are identified with their image.  As elements of $End(\V)\otimes \Rep(G,M)$, they can be evaluated on a group action map $\rho:(G,M)\rightarrow (SL_2(\k),\V)$ to give elements of $End(\V)$.  Specifically,
\[ \X_g(\rho) = \rho(g),\;\;\;\;\;\oute{p}{q}(\rho) = \Theta(\rho(p),\rho(q)) = \rho(p)\rho(q)^\perp\]
The anti-involution $\iota:\mathcal{T}_GM^2\rightarrow \mathcal{T}_GM^2$ can be formally defined by
\[ (\X_g)^\iota := \X_{g^{-1}},\;\;\;\;\; (\oute{p}{q})^\iota = -\oute{q}{p}\]
Similarly, the linear map $tr:\mathcal{T}_GM^2\dashrightarrow \mathcal{T}_GM^2$ can be formally defined, by
\[ tr(\mathbf{A}):=\mathbf{A}+\mathbf{A}^\iota\]


Then, the following elegant theorem is the fruit of all of the labor so far, and the source of all of the results to follow.

\begin{thm}\label{thm: matcharpres}
By the surjection \eqref{eq: tensorsurj}, the algebra $(End(\V)\otimes \Rep(G,M))^{SL_2}$ is isomorphic to the quotient of $\mathcal{T}_GM^2$ by the two-sided ideal generated by
\begin{enumerate}
\item $tr(\mathbf{A})\mathbf{B}-\mathbf{B}tr(\mathbf{A})$, for all $\mathbf{A},\mathbf{B}\in \mathcal{T}_GM^2$.
\item $\oute{p}{q}\oute{p'}{q'}-tr(\oute{p'}{q})\oute{p}{q'}$,
for all $p,q,p',q'\in M$.
\end{enumerate}
\end{thm}
\begin{proof}
Quotienting by just first three relations in Corollary \ref{coro: matcharpres} gives $\mathcal{T}_GM^2$, by the lemma.    The remaining two relations are almost copied directly from Corollary \ref{coro: matcharpres}; the only difference is that the Relation (2) no longer contains a generic word $\mathbf{A}$.

Note that it suffices to assume that such an $\mathbf{A}$ is a word in $\{\oute{p}{q}\}$, since any $\X_g$ may be absorbed into an outer product $\oute{p}{q}$.  For $\mathbf{A}=\oute{p'}{q'}$, using Relation (5),
\begin{eqnarray*}
 \oute{p}{q}(\oute{p'}{q'})\oute{p''}{q''} &=&  tr(\oute{p'}{q})\oute{p}{q'}\oute{p''}{q''}\\
&=&tr(\oute{p'}{q})tr(\oute{p''}{q'})\oute{p}{q''}\\
&=& tr((\oute{p'}{q'})\oute{p''}{q})\oute{p}{q''}\\
\end{eqnarray*}
The same argument works for longer $\mathbf{A}$, and so every relation of type (5) in Corollary \ref{coro: matcharpres} can be deduced from Relation (2) in the statement of the theorem.
\end{proof}

\subsection{A presentation of the character algebra}  Recall the map
\[ \chi:H^+(G,M)\rightarrow \Rep(G,M)^{SL_2}\]
which we wish to show is an isomorphism.  Composing with the inclusion $\Rep(G,M)\hookrightarrow End(\V)\otimes \Rep(G,M)$ gives
\[ \chi:H^+(G,M)\rightarrow (End(\V)\otimes \Rep(G,M))^{SL_2}\]
In this larger context, $\chi$ may be expressed in terms of the trace.
\[ \chi([g]) = tr(\X_g),\;\;\;\;\chi([p,q]) = tr(\oute{q}{p})\]

The scalars $\k\subset End(\V)$ are characterized as the elements which are $\iota$-invariant.  By extension, the representation algebra
\[ \Rep(G,M)\subset End(\V)\otimes \Rep(G,M)\]
is the $\iota$-invariant subalgebra, and so
\[ \Rep(G,M)^{SL_2}\subset (End(\V)\otimes \Rep(G,M))^{SL_2}\]
is the $\iota$-invariant subalgebra.  Since the trace map projects onto the $\iota$-invariants, $\Rep(G,M)^{SL_2}$ is the subalgebra of $(End(\V)\otimes \Rep(G,M))^{SL_2}$ which is the image of the trace.


Define a linear map
$ \tau:\mathcal{T}_GM^2\dashrightarrow H^+(G,M)$
by
$ \tau(\X_g) := [g]$
and
\[ \tau\left(\oute{p_1}{q_1}\oute{p_2}{q_2}...\oute{p_n}{q_n}\right):=[q_1,p_2][q_2,p_3]...[q_n,p_1]\]
The following proposition collects the important properties of $\tau$.
\begin{prop} Let $\mathbf{A},\mathbf{B}\in \mathcal{T}_GM^2$.
\begin{enumerate}
\item $\tau(\mathbf{A}^\iota)=\tau(\mathbf{A})$.
\item $\tau(\mathbf{A}\mathbf{B})=\tau(\mathbf{B}\mathbf{A})$.
\item $\tau(tr(\mathbf{A})\mathbf{B})=\tau(\mathbf{A})\tau(\mathbf{B})$.
\item The map $\tau$ descends to a map $\tau:(End(\V)\otimes \Rep(G,M))^{SL_2}\dashrightarrow H^+(G,M)$.
\end{enumerate}
\end{prop}
\begin{proof}
\begin{enumerate}
\item One has
$ [g]+[g^{-1}]=[e][g]=2[g]$
and so $[g^{-1}]=[g]$.
\[ \tau(\X_g) = [g]=[g^{-1}] = \tau(\X_{g^{-1}})\]
Next, by definition, $[p,q]=-[q,p]$, and so
\begin{eqnarray*}
\tau\left((\oute{p_1}{q_1}\oute{p_2}{q_2}...\oute{p_n}{q_n})^\iota\right)
&=&\tau\left(\oute{p_n}{q_n}^\iota\oute{p_{n-1}}{q_{n-1}}^\iota...\oute{p_1}{q_1}^\iota\right) \\
&=&(-1)^n\tau\left(\oute{q_n}{p_n}\oute{q_{n-1}}{p_{n-1}}...\oute{q_1}{p_1}\right) \\
&=&(-1)^n[p_n,q_{n-1}][p_{n-1},q_{n-2}]...[p_1,q_n] \\
&=&[q_1,p_2][q_2,p_3]...[q_n,p_1]\\
&=&\tau\left(\oute{p_1}{q_1}\oute{p_2}{q_2}...\oute{p_n}{q_n}\right)
\end{eqnarray*}
\item One has
\[ [gh] = [g][h]-[h^{-1}g] = [h][g]-[g^{-1}h] = [hg]\]
and so $\tau(\X_{g}\X_h)=[gh]=[hg]=\tau(\X_{h}\X_g)$.
\begin{eqnarray*}
\tau(\X_g\oute{p_1}{q_1}\oute{p_2}{q_2}...\oute{p_n}{q_n}) &=& [q_1,p_2][q_2,p_3]...[q_n,gp_1]\\
&=& [q_1,p_2][q_2,p_3]...[g^{-1}q_n,p_1]\\
&=& \tau(\oute{p_1}{q_1}\oute{p_2}{q_2}...\oute{p_n}{q_n}\X_g)
\end{eqnarray*}
\begin{eqnarray*}
\tau(\oute{p_1}{q_1}...\oute{p_i}{q_i}\oute{p_{i+1}}{q_{i+1}}...\oute{p_n}{q_n}) &=& [q_1,p_2]...[q_i,p_{i+1}][q_{i+1},p_{i+2}]...[q_n,p_1]\\
&=& [q_{i+1},p_{i+2}]...[q_n,p_1][q_1,p_2]...[q_i,p_{i+1}]\\
&=& \tau(\oute{p_{i+1}}{q_{i+1}}...\oute{p_n}{q_n}\oute{p_1}{q_1}...\oute{p_i}{q_i})
\end{eqnarray*}
Thus, $\tau(\mathbf{A}\mathbf{B})=\tau(\mathbf{B}\mathbf{A})$.

\item
%
We check all the cases.
\begin{eqnarray*}
\tau(tr(\X_g)\X_h) &=& \tau(\X_{gh}+\X_{g^{-1}h}) = [gh]+[g^{-1}h]=[g][h]=\tau(\X_g)\tau(\X_h)
\end{eqnarray*}
\begin{eqnarray*}
\tau(tr(\X_g)\oute{p_1}{q_1}\oute{p_2}{q_2}...\oute{p_n}{q_n}) &=& \tau((\oute{gp_1}{q_1}+\oute{g^{-1}p_1}{q_1})\oute{p_2}{q_2}...\oute{p_n}{q_n}) \\
&=& [q_1,p_2]...[q_{n-1},p_n]([q_n,gp_1]+[q_n,g^{-1}p_1])\\
&=& [g][q_1,p_2]...[q_{n-1},p_n][q_n,p_1]\\
&=& \tau(\X_g)\tau(\oute{p_1}{q_1}\oute{p_2}{q_2}...\oute{p_n}{q_n})
\end{eqnarray*}
\begin{align*}
 \tau (tr & (\oute{p_1}{q_1}... \oute{p_j}{q_j})   \oute{p_{j+1}}{q_{j+1}}...\oute{p_n}{q_n} )   \\
= & \tau(\oute{p_1}{q_1}... \oute{p_j}{q_j} \oute{p_{j+1}}{q_{j+1}}...\oute{p_n}{q_n}
+(-1)^j\oute{q_j}{p_j}... \oute{q_1}{p_1} \oute{p_{j+1}}{q_{j+1}}...\oute{p_n}{q_n} ) \\
= & [q_1,p_2]...[q_j,p_{j+1}]...[q_n,p_1] + (-1)^{j}[p_j,q_{j-1}]...[p_1,p_{j+1}]...[q_n,q_j] \\
= & [q_1,p_2]...[q_{j-1},p_j][q_{j+1},p_{j+2}]...[q_{n-1},p_n]([q_j,p_{j+1}][q_n,p_1] -[p_1,p_{j+1}][q_n,q_j]) \\
= & ([q_1,p_2]...[q_{j-1},p_j][q_j,p_1])([q_{j+1},p_{j+2}]...[q_{n-1},p_n][q_n,p_j]) \\
= & \tau  (\oute{p_1}{q_1}... \oute{p_j}{q_j})\tau(   \oute{p_{j+1}}{q_{j+1}}...\oute{p_n}{q_n} )
\end{align*}

\item Let $\mathbf{A},\mathbf{B},\mathbf{C},\mathbf{D}\in \mathcal{T}_GM^2$.  Then
\begin{eqnarray*}
\tau(\mathbf{C}tr(\mathbf{A})\mathbf{B}\mathbf{D}) &=& \tau(tr(\mathbf{A})\mathbf{B}\mathbf{D}\mathbf{C}) \\
&=& \tau(\mathbf{A})\tau(\mathbf{B}\mathbf{D}\mathbf{C}) \\
&=& \tau(\mathbf{A})\tau(\mathbf{D}\mathbf{C}\mathbf{B}) \\
&=& \tau(tr(\mathbf{A})\mathbf{D}\mathbf{C}\mathbf{B}) \\
&=& \tau(\mathbf{C}\mathbf{B}tr(\mathbf{A})\mathbf{D})
\end{eqnarray*}
Therefore, $\tau$ kills $\mathbf{C}(tr(\mathbf{A})\mathbf{B}-\mathbf{B}tr(\mathbf{A}))\mathbf{D}$.  Next, for $\{p_i\},\{q_i\}\in M$,
\begin{align*}
\tau(\oute{p_1}{q_1}... \oute{p_j}{q_j}  & \oute{p_{j+1}}{q_{j+1}}...\oute{p_n}{q_n} )   \\
= & [q_1,p_2]....[q_{j-1},p_j][q_j,p_{j+1}]...[q_n,p_1] \\
= & \tau(\oute{p_{j+1}}{q_j})\tau\left(\oute{p_1}{q_1}...\oute{p_j}{q_{j+1}}...\oute{p_n}{q_n}\right)\\
= & \tau(\oute{p_{j+1}}{q_j})\tau\left(\oute{p_j}{q_{j+1}}...\oute{p_n}{q_n}\oute{p_1}{q_1}...\oute{p_{j-1}}{q_j}\right)\\
= & \tau\left(tr(\oute{p_{j+1}}{q_j})\oute{p_j}{q_{j+1}}...\oute{p_n}{q_n}\oute{p_1}{q_1}...\oute{p_{j-1}}{q_j}\right)\\
= & \tau\left(\oute{p_1}{q_1}...\oute{p_{j-1}}{q_j}tr(\oute{p_{j+1}}{q_j})\oute{p_j}{q_{j+1}}...\oute{p_n}{q_n}\right)
\end{align*}
Therefore, $\tau$ kills $\mathbf{A}(\oute{p}{q}\oute{p'}{q'}-tr(\oute{q}{p'})\oute{p}{q'})\mathbf{B}$.  By Theorem \ref{thm: matcharpres}, the map $\tau$ descends to a map $(End(\V)\otimes \Rep(G,M))^{SL_2}\dashrightarrow H^+(G,M)$.
\end{enumerate}
\end{proof}


We can now show that $\chi:H^+(G,M)\rightarrow \Rep(G,M)^{SL_2}$ is an isomorphism.
\begin{proof}[Proof of Theorem \ref{thm: main}] First, we show $\tau\chi(r)=2r$ for all $r$ by inducting on the length $l$ of $r\in H^+(G,M)$.  First, $\tau\chi(1)=\tau(1) =[e]=2$.  Next, let $r\in H^+(G,M)$ be a word of length $l-1$.
\[ \tau\chi([g]r) = \tau(tr(\X_g)\chi(r))=\tau(\X_g)\tau\chi(r)=[g]\tau\chi(r)=2[g]r\]
\[ \tau\chi([p,q]r) = \tau(tr(\oute{q}{p})\chi(r))=\tau(\oute{q}{p})\tau\chi(r)=[p,q]\tau\chi(r)=2[p,q]r\]
Therefore, $\chi$ is an inclusion.

Next, we show that $\chi\tau(A)=tr(A)$ for all $A\in End(\V)\otimes \Rep(G,M)^{SL_2}$.
\[ \chi\tau(\X_g) = \chi([g]) = tr(\X_g)\]
\begin{eqnarray*}
\chi\tau\left(\oute{p_1}{q_1}\oute{p_2}{q_2}...\oute{p_n}{q_n}\right)&=&\chi([q_1,p_2][q_2,p_3]...[q_n,p_1])\\
&=& tr(\oute{p_2}{q_1})tr(\oute{p_3}{q_2})...tr(\oute{p_1}{q_n})\\
&=& tr(\oute{p_1}{q_1}\oute{p_2}{q_2}...\oute{p_n}{q_n})
\end{eqnarray*}
Therefore, the image of $\chi$ is $\Rep(G,M)^{SL_2}$.
\end{proof}


\appendix

\section{Twisted character algebras}
The preceeding theory concerned group action maps $(G,M)\rightarrow (SL_2(\k),\V)$.  A useful variant is to consider group action maps which are `twisted' by a $\mathbb{Z}_2$-character of a central extension.  This is important for applications to Teichm\"uller space, skein algebras, and cluster algebras.

%

%
%
%

\subsection{Twisting by a central extension}

Let $(G,M)$ be a group action.  A \textbf{central extension} of $(G,M)$ will be a group action $(G',M')$, together with a map
\[ f=(f_G,f_M):(G',M')\rightarrow (G,M)\]
such that...
\begin{itemize}
\item $f_G$ and $f_M$ are surjective,
\item the kernel $K$ of $f_G$ is central in $G'$, and
\item for each $m\in M$, the preimage $f_M^{-1}(m)\subset M'$ is a (non-empty) free $K$-orbit.
\end{itemize}
A central extension of groups $G'\rightarrow G$ determines a central extension of group actions up to non-canonical isomorphism.

Let $s:K\rightarrow \{\pm1\}$ be a group homomorphism.\footnote{A choice of central extension $(G',M')$ is implicit in the definition of $s$.}  Then an \textbf{$s$-twisted representation} of $(G,M)$ into $(SL_2(\k),\V)$ is a map of group actions $\rho:(G',M')\rightarrow (SL_2(\k),\V)$ such that, for all $k\in K$, $\rho(k)=s(k)$.  If $s(K)=1$, then an $s$-twisted representation is equivalent to a(n untwisted) representation $(G,M)\rightarrow (SL_2(\k),\V)$.  Similarly, a splitting $h:(G,M)\rightarrow (G',M')$ (i.e., $f\circ h$ is the identity) induces a bijection between $s$-twisted representations and untwisted representations by pulling back along $h$.

\subsection{The twisted character algebra} The set of $s$-twisted representations of $(G,M)$ into $(SL_2(\k),\V)$ has a natural algebraic structure.  Define the \textbf{$s$-twisted representation algebra} as
\[ \Rep_s(G,M):=\Rep(G',M')/\langle g-s(g)\rangle_{g\in K}\]
This algebra has the following universal property; this follows from Proposition \ref{prop: unirep}.
\[ \left\{ \begin{array}{c}\text{group action maps}\\ \rho:(G',M')\rightarrow (SL_2(A),A\otimes \V)\\ \text{such that }\rho(g)=s(g),\forall g\in K\end{array}\right\}
\stackrel{\sim}{\longrightarrow}
\left\{ \begin{array}{c}\text{$\k$-algebra maps}\\ \Rep_s(G,M)\rightarrow A\end{array}\right\}\]
This algebra has an $SL_2(\k)$-action.  Define the \textbf{twisted character algebra} 
 to be the $SL_2$-invariant subalgebra,
\[ \Char_s(G,M):=\Rep_s(G,M)^{SL_2}\]
The definition of $Rep_s(G,M)$ provides a surjection
\[ \Rep(G',M')\rightarrow \Rep_s(G,M),\]
which induces a map
\[ \mu:\Char(G',M')\rightarrow \Char_s(G,M)\]
\begin{lemma}
The map $\mu$ is surjective, and the kernel of $\mu$ is generated by elements of the form 
\begin{itemize}
\item $\chi_{gh}-s(g)\chi_h$, for $g\in K$ and $h\in G'$, and
\item $\chi_{(gp,q)}-s(g)\chi_{(p,q)}$, for $g\in K$ and $p,q\in M'$.
\end{itemize}
\end{lemma}
\begin{proof}[Proof outline]
This proof is in the same spirit as the presentation of the character algebra, so we only outline the details.  First, the map on representation algebras induces a surjection
\[ End(\V)\otimes \Rep(G',M')\rightarrow End(\V)\otimes \Rep_s(G,M)\]
One observes that the kernel is generated by $\X_g-s(g)Id_\V$ as $g$ runs over $K$.  These relations are $SL_2$-invariant, and so by Lemma \ref{lemma: invrel}, the kernel of 
\[ (End(\V)\otimes \Rep(G',M'))^{SL_2}\rightarrow (End(\V)\otimes \Rep_s(G,M))^{SL_2}\]
is also generated by $\X_g-s(g)Id$ as $g$ runs over $K$.

Since the corresponding character algebras are the $\iota$-invariant subalgebras in this map, the kernel of $\mu$ is spanned by elements of the form 
\[tr(\mathbf{A}(\X_g-s(g)Id)\mathbf{B})=tr( (\X_g-s(g))\mathbf{BA})=tr(\X_g\mathbf{BA})-s(g)tr(\mathbf{BA})\]
From this, the theorem may be deduced directly.
\end{proof}
Then $\Char_s(G,M)$ can be presented by using Theorem \ref{thm: main}, with the additional classes of relations $\{\chi_{gh}=s(g)\chi_h\}$ and $\{\chi_{(gp,q)}=s(g)\chi_{(p,q)}\}$.

\bibliography{MyNewBib}{}

\def\cprime{$'$} \def\cprime{$'$} \def\cprime{$'$} \def\cprime{$'$}
  \def\cprime{$'$}
\providecommand{\bysame}{\leavevmode\hbox to3em{\hrulefill}\thinspace}
\providecommand{\MR}{\relax\ifhmode\unskip\space\fi MR }
\providecommand{\MRhref}[2]{%
  \href{http://www.ams.org/mathscinet-getitem?mr=#1}{#2}
}
\providecommand{\href}[2]{#2}
\begin{thebibliography}{GSV05}

\bibitem[Art69]{Art69}
M.~Artin, \emph{On {A}zumaya algebras and finite dimensional representations of
  rings.}, J. Algebra \textbf{11} (1969), 532--563. \MR{0242890 (39 \#4217)}

\bibitem[Ati90]{Ati90}
Michael Atiyah, \emph{The geometry and physics of knots}, Lezioni Lincee.
  [Lincei Lectures], Cambridge University Press, Cambridge, 1990. \MR{1078014
  (92b:57008)}

\bibitem[Bar99]{Bar99}
John~W. Barrett, \emph{Skein spaces and spin structures}, Math. Proc. Cambridge
  Philos. Soc. \textbf{126} (1999), no.~2, 267--275. \MR{1670233 (99k:57006)}

\bibitem[BH95]{BH95}
G.~W. Brumfiel and H.~M. Hilden, \emph{{${\rm SL}(2)$} representations of
  finitely presented groups}, Contemporary Mathematics, vol. 187, American
  Mathematical Society, Providence, RI, 1995. \MR{1339764 (96g:20004)}

\bibitem[Bul97]{Bul97}
Doug Bullock, \emph{Rings of {${\rm SL}_2({\bf C})$}-characters and the
  {K}auffman bracket skein module}, Comment. Math. Helv. \textbf{72} (1997),
  no.~4, 521--542. \MR{1600138 (98k:57008)}

\bibitem[CS83]{CS83}
Marc Culler and Peter~B. Shalen, \emph{Varieties of group representations and
  splittings of {$3$}-manifolds}, Ann. of Math. (2) \textbf{117} (1983), no.~1,
  109--146. \MR{683804 (84k:57005)}

\bibitem[FG06]{FG06}
Vladimir Fock and Alexander Goncharov, \emph{Moduli spaces of local systems and
  higher {T}eichm\"uller theory}, Publ. Math. Inst. Hautes \'Etudes Sci.
  (2006), no.~103, 1--211. \MR{2233852 (2009k:32011)}

\bibitem[FST08]{FST08}
Sergey Fomin, Michael Shapiro, and Dylan Thurston, \emph{Cluster algebras and
  triangulated surfaces. {I}. {C}luster complexes}, Acta Math. \textbf{201}
  (2008), no.~1, 83--146. \MR{2448067 (2010b:57032)}

\bibitem[GSV05]{GSV05}
Michael Gekhtman, Michael Shapiro, and Alek Vainshtein, \emph{Cluster algebras
  and {W}eil-{P}etersson forms}, Duke Math. J. \textbf{127} (2005), no.~2,
  291--311. \MR{2130414 (2006d:53103)}

\bibitem[How89]{How89}
Roger Howe, \emph{Remarks on classical invariant theory}, Trans. Amer. Math.
  Soc. \textbf{313} (1989), no.~2, 539--570. \MR{986027 (90h:22015a)}

\bibitem[MS]{MSSkein}
G.~Muller and P.~Samuelson, \emph{Skein algebras and decorated {${\rm
  SL}_2({\bf C})$}-local systems on surfaces}, in preparation.

\bibitem[Pro76]{Pro76}
C.~Procesi, \emph{The invariant theory of {$n\times n$} matrices}, Advances in
  Math. \textbf{19} (1976), no.~3, 306--381. \MR{0419491 (54 \#7512)}

\bibitem[Pro84]{Pro84}
\bysame, \emph{Computing with {$2\times 2$} matrices}, J. Algebra \textbf{87}
  (1984), no.~2, 342--359. \MR{739938 (86g:16022)}

\bibitem[PS00]{PS00}
J{\'o}zef~H. Przytycki and Adam~S. Sikora, \emph{On skein algebras and {${\rm
  Sl}_2({\bf C})$}-character varieties}, Topology \textbf{39} (2000), no.~1,
  115--148. \MR{1710996 (2000g:57026)}

\bibitem[Stu08]{Stu08}
Bernd Sturmfels, \emph{Algorithms in invariant theory}, second ed., Texts and
  Monographs in Symbolic Computation, SpringerWienNewYork, Vienna, 2008.
  \MR{2667486}

\bibitem[Sza09]{Sza09}
Tam{\'a}s Szamuely, \emph{Galois groups and fundamental groups}, Cambridge
  Studies in Advanced Mathematics, vol. 117, Cambridge University Press,
  Cambridge, 2009. \MR{2548205 (2011b:14064)}

\bibitem[Wey39]{Wey39}
Hermann Weyl, \emph{The {C}lassical {G}roups. {T}heir {I}nvariants and
  {R}epresentations}, Princeton University Press, Princeton, N.J., 1939.
  \MR{0000255 (1,42c)}

\end{thebibliography}
\bibliographystyle{amsalpha}

\end{document}